\documentclass[11pt,a4paper]{article}

\usepackage{amsfonts, amsmath, amssymb, amsthm, enumerate}
\usepackage{psfrag,graphicx,xcolor}
\newtheorem{proposition}{Proposition}[section]
\newtheorem{theorem}[proposition]{Theorem}
\newtheorem{remark}[proposition]{Remark}
\newtheorem{lemma}[proposition]{Lemma}

\newtheorem{definition}[proposition]{Definition}

\newcommand\sqr[2]{{\vcenter{\vbox{\hrule height.#2pt
   \hbox{\vrule width.#2pt height#1pt \kern#1pt
      \vrule width.#2pt}
   \hrule height.#2pt}}}}

\newcommand{\ds}{\displaystyle}

\renewcommand{\epsilon}{\varepsilon}
\newcommand{\eps}{\epsilon}

\newcommand{\BV}{\mathbf{BV}}

\newcommand{\C}[1]{{{\mathbf{C}^\mathbf{#1}}}}

\renewcommand{\L}[1]{{{\mathbf{L}^\mathbf{#1}}}}
\newcommand{\Lloc}[1]{\mathbf{L}^{\mathbf{1}}_{\mathrm{loc}}}

\newcommand{\reali}{{\mathbb{R}}}

\newcommand{\naturali}{\mathbb{N}}

\newcommand{\Id}{I\!d}

\newcommand{\RII}{R_{I\!I}}
\newcommand{\RLL}{R_{L\!L}}
\newcommand{\RLI}{R_{L\!I}}
\newcommand{\RIL}{R_{I\!L}}

\newcommand{\xa}{\mathpzc{a}}
\newcommand{\xb}{\mathpzc{b}}

\DeclareMathAlphabet{\mathpzc}{OT1}{pzc}{m}{it}

\newcommand{\sgn}{\mathrm{sgn}}

\newcommand{\tv}{\mathop{\rm TV}}

\makeatletter

\newlength{\captionwidth}
\long\def\@makecaption#1#2{%
   \vskip 10\p@
   \setbox\@tempboxa\hbox{#1: #2}%
   \ifdim \wd\@tempboxa > \captionwidth 
       \hbox to\hsize{\hfil
       \parbox[t]{\captionwidth}{
       #1: #2\par}
       \hfil}
     \else
       \hbox to\hsize{\hfil\box\@tempboxa\hfil}%
   \fi}

\makeatother

\title{Global existence of solutions for a multi-phase flow:\\
 a drop in a gas-tube}

\author{}

\author{Debora Amadori\footnote{Department of Engineering and Computer Science and Mathematics, University of L'Aquila,
Italy}
\and Paolo Baiti\footnote{Department of Mathematics and Computer Science, University of Udine, Italy}
\and Andrea Corli\footnote{Department of Mathematics and Computer Science,
University of Ferrara, Italy}
\and Edda Dal Santo\footnotemark[2]
}

\date{February 6, 2015}

\begin{document}

\maketitle

\begin{abstract}
In this paper we study the flow of an inviscid fluid composed by three different phases. The model is a simple hyperbolic system of three conservation laws, in Lagrangian coordinates, where the phase interfaces are stationary. Our main result concerns the global existence of weak entropic solutions to the initial-value problem for large initial data.
\end{abstract}

\smallskip

\noindent\textit{2010~Mathematics Subject Classification:} 35L65,
35L60, 35L67, 76T99.

\smallskip



\section{Introduction}\label{sec:intro}


The theory of hyperbolic systems of conservation laws in one spatial dimension has reached in the last years a rather satisfactory level of
completeness, as the reference book of Dafermos \cite{Dafermos} witnesses. Among the several important results that have been proved, probably the greatest achievement concerns the global existence in time of weak solutions to the initial-value problem, as well as their uniqueness, continuity with respect to the data and viscous approximations. However, such results hold, in general, only for {\em small} initial data: the case of {\em large} data has been given no general and satisfactory answer. This paper focuses precisely on this issue in the case of a simple but physically meaningful system of three equations, for which we provide {\em explicit} conditions on the initial data in order to have global solutions.

The system under consideration arises in the modeling of phase transitions for an inviscid fluid and is deduced by \cite{Fan}. If we denote by $v>0$ the specific volume of the fluid, $u$ the velocity, $p$ the pressure and $\lambda$ the mass-density fraction of the vapor, the system is written as
\begin{equation}\label{eq:system}
\left\{
\begin{array}{ll}
v_t - u_x &= 0\,,\\
u_t + p(v,\lambda)_x &= 0\,,
\\
\lambda_t &= 0\,.
\end{array}
\right.
\end{equation}
As usual, here $t>0$ denotes the time and $x\in\reali$. The phase states of the fluid are modeled by the variable $\lambda$, which ranges from $0$ (pure liquid) to $1$ (pure vapor) and allows for intermediate values in the interval $]0,1[$ representing mixtures of the two pure phases. The model incorporates the state variable $\lambda$ in the pressure, which is defined by
\begin{equation}\label{eq:pressure}
p(v,\lambda)= \frac{a^2(\lambda)}v\,,
\end{equation}
where $a(\lambda)>0$ and is a $\C{1}$ function on $[0,1]$. We denote by $U=(v,u,\lambda)$ the state variables and by $\Omega=]0,+\infty[ \times \reali \times [0,1]$ the state space. System \eqref{eq:system} is strictly hyperbolic in $\Omega$ with eigenvalues $e_1 = -\sqrt{-p_v}$, $e_2 = 0$, $e_3 = \sqrt{-p_v}$; the first and the third characteristic fields are genuinely nonlinear, while the second one is linearly degenerate.

The first result on the existence of global solutions to system \eqref{eq:system}, provided with suitably large initial data, is given in \cite{amadori-corli-siam}; a different proof is given in \cite{Asakura-Corli}. In particular, in the case where $\lambda$ is constant, the classical result by Nishida \cite{Nishida68} is recovered.

The analysis of \cite{amadori-corli-siam} is pursued and refined in \cite{ABCD} for initial data with $\lambda$ of Riemann type:
$\lambda(x,0)$ is constant for $x\ne0$ with a jump at $0$. In this case system \eqref{eq:system} decouples for any $t>0$
into two $p$-systems connected by a phase interface at $x=0$, because the discontinuities of $\lambda$ do not propagate. %
The special form of $\lambda$ allowed us to analyze in detail the effect of the nonlinear interaction of pressure waves through the phase interface,
leading to refined sufficient conditions on the initial data for which solutions exist globally in time.

A survey on couplings of two systems of conservation laws, with a focus on numerical approximations, is given in \cite{Godlewski_survey};
however, we emphasize that the above coupling for system \eqref{eq:system} is {\em the} physical coupling, where the interface is a contact
discontinuity. We refer to \cite{amadori-corli-siam, ABCD} for further references on related results.

In this paper we continue the analysis of system \eqref{eq:system} by considering the case where the initial datum for $\lambda$
is piecewise constant with {\em two} jumps. 
Let the initial data be of the form
\begin{equation}
U_o(x) = \left(v_o(x), u_o(x), \lambda_o(x) \right)\,, \quad \hbox{ for }\quad \lambda_o(x) = \left\{
\begin{array}{ll}
\lambda_\ell & \hbox{ if }x<\xa \,,\\
\lambda_m & \hbox{ if }\xa<x<\xb\,,\\
\lambda_r & \hbox{ if }x>\xb\,,\\
\end{array}
\right.
\label{init-data}
\end{equation}
where $x\in\reali$ and $\lambda_\ell$, $\lambda_m$, $\lambda_r\in[0,1]$ are constant. We define $a_\ell=a(\lambda_\ell)$, $a_m=a(\lambda_m)$, $a_r=a(\lambda_r)$ and focus on the case
\begin{equation}\label{eq:case}
a_m<\min\{a_\ell,a_r\}\,.
\end{equation}
To give a flavor of the physical meaning of the problem, assume that $a(\lambda)$ is \emph{increasing} 
(which is the physically meaningful case) so that \eqref{eq:case} implies $\lambda_m < \min\left\{\lambda_\ell,\lambda_r\right\}$.
Here, we are dealing with a one-dimensional fluid consisting of three homogeneous mixtures of liquid and vapor;
the mixture in the region $]\xa,\xb[$ is more liquid than in the surrounding ones. This includes the case of a liquid drop in a gaseous environment.
The other cases $a_m>\max\{a_\ell,a_r\}$ (a bubble surrounded by liquid) and $a_\ell<a_m<a_r$ (or $a_\ell>a_m>a_r$) are considered
in a forthcoming work \cite{ABCD3}.

A similar model is studied in \cite{ColomboSchleper}. There, the basic system (in Eulerian coordinates) has only two equations but is augmented with
kinetic conditions deduced by the mass and momentum conservation at the interfaces, which make that model essentially equivalent to
\eqref{eq:system}. However, the results of \cite{ColomboSchleper} concern a general pressure law but {\em small} initial data. We refer also to
\cite{ColomboGuerraSchleper} where an analogous system (in Lagrangian coordinates) is studied, in which the pressure in the region $[\xa,\xb]$ is a linear function of $v$.

Notice also that \eqref{eq:system}, \eqref{init-data} can be interpreted as a perturbation problem of the steady solution
given by the two parallel contact discontinuities located at $x=\xa$ and $x=\xb$, respectively.
We refer to 
\cite{LewickaBV,Schochet-Glimm} for the analysis of the perturbation 
of a single contact discontinuity.

The main result of this paper is Theorem~\ref{thm:main}, that provides a wide class of {\em large} initial data
for which the solution to the initial-value problem \eqref{eq:system}, \eqref{init-data} exists globally in time.
Roughly speaking, the conditions on the data require that the total variations of $p_o=p(v_o,\lambda_o)$ and $u_o$ do not exceed a certain threshold
depending on the sizes of the interfaces, see \eqref{hyp2}: the larger are the interfaces, the smaller must be the variations and conversely.
Also, if the variations are sufficiently small then {\em any} size of the interfaces is allowed, provided that
the stability condition \eqref{stab} (that was missing in \cite{ABCD}) holds.
Such a result was proved, to the best of our knowledge, in no related paper.
Moreover, we point out that our conditions on the initial data are sufficiently flexible to allow the control of the variations in either of the three phases.

For the proof of Theorem~\ref{thm:main}, two novel ideas are employed.
The first one is a simplification in the definition of the functional $F$ used to control the total variation of the solutions (see \eqref{F} and Remark~\ref{rem:F_asymm}), in which some nonlinear terms are dropped, thanks to a more careful use of nonlinear interactions involving phase waves.

The second one is an original variant of the front-tracking algorithm \cite{Bressanbook}, that is needed in order to ensure
that the functional $F$ is decreasing.
Indeed, the classical front-tracking scheme prescribes two ways of solving the Riemann problem arising at an interaction:
either by means of an Accurate solver
or by a Simplified solver that exploits \emph{non-physical waves}, which are used to prevent the possible blow-up in
finite time of the number of fronts and interactions.
Here we provide an original definition of the Simplified solver, suitably designed for this problem.
At any interaction of a small wave with an interface, the Simplified solver introduces stationary, non-entropic waves
(associated to the integral curves of \eqref{eq:system}), which are formally computed as reflected waves.
These waves ``travel'' with zero speed and then remain attached to the phase wave, thus forming a ``composite wave''.
Such a Riemann solver somewhat reminds of the famous Osher solver frequently used in Numerical Analysis, see \cite[\S 12]{Toro}.
The idea of introducing stationary composite waves for the Simplified solver is also exploited in \cite{ABCD} where,
however, the jump across non-physical waves is defined as in \cite{Bressanbook}.

Notice that when $\lambda$ is constant, system \eqref{eq:system} reduces to a $2\times 2$ system and one can avoid the use of the Simplified solver
(see \cite{AmadoriGuerra01, BaitiDalSanto}).
We point out that non-physical waves are also avoided in
\cite{ColomboSchleper,ColomboGuerraSchleper}, but for different reasons: in \cite{ColomboSchleper}, 
due to a particular solver and to the smallness of the data, while in \cite{ColomboGuerraSchleper}
because of the assumption of linear pressure in the region $[\xa,\xb]$.

The paper is organized as follows. The main result is stated in Section~\ref{main}. In Section~\ref{sec:Riemann} we first introduce four pre-Riemann
solvers: one of them is used to define the composite wave, the other three are exploited either in the Accurate or in the Simplified solver. 
Proposition~\ref{prop:accsimpl} gives a unified approach to both the Accurate and the Simplified solver.
Approximate solutions are defined in
Section~\ref{sec:app_sol}. In Section~\ref{sec:interaction} 
we introduce the main functional $F$ and show that it is decreasing in time.
Section~\ref{sec:convergence} deals with the convergence and consistency of the algorithm, together with a decay property of the reflected waves;
we provide also a comparison with \cite{amadori-corli-siam} which shows how the preceding result is
improved. Finally, in Appendix~\ref{appB} it is proved an alternative estimate concerning certain interactions solved by the Simplified solver.


\section{Main Result}\label{main}
\setcounter{equation}{0}

Throughout this paper we assume \eqref{eq:case} and call $\eta$, $\zeta$ the strengths of the two $2$-waves, as 
in \cite{AmCo06-Proceed-Lyon}:
$$
 \eta = 2\, \frac{a_m-a_\ell}{a_m+a_\ell}\,,\qquad \zeta=2\, \frac{a_r-a_m}{a_r+a_m}\,.
$$
By \eqref{eq:case} and for $a_\ell$, $a_m$, $a_r$ in $\reali^+=\,]0,+\infty[$, one easily finds that
$$
\eta<0\,,\qquad \zeta>0\,,\qquad |\eta|,\,|\zeta| \in [0,2[ \,.
$$   
In what follows, we need that $\eta$ and $\zeta$ satisfy the \emph{stability} condition
\begin{equation}\label{stab}
    \max\left\{\bigl(1+\frac{|\zeta|}{2}\bigr)\frac{|\eta|}{2},
    \bigl(1+\frac{|\eta|}{2}\bigr)\frac{|\zeta|}{2}\right\}<1\,.
\end{equation}
When one of the two waves $\eta$ or $\zeta$ vanishes, for example $\zeta=0$, then \eqref{stab} reduces to $|\eta|<2$, which is always satisfied.
The inequality \eqref{stab} identifies a set $\mathcal{D} \subset [0,2[\times [0,2[$ (see Figure~\ref{ngoccia2}), where we define a non-negative and
continuous function $\mathcal{H}$ by
\begin{equation}\label{eq:Hdef}
\mathcal{H}(|\eta|,|\zeta|)=
\max\left\{\frac{|\zeta|}{1-(1+{|\zeta|}/{2}){|\eta|}/{2}},
\frac{|\eta|}{1-(1+{|\eta|}/{2}){|\zeta|}/{2}}\right\}\,.
\end{equation}

\begin{figure}[htbp]
    \begin{center}
	\includegraphics[width=6cm]{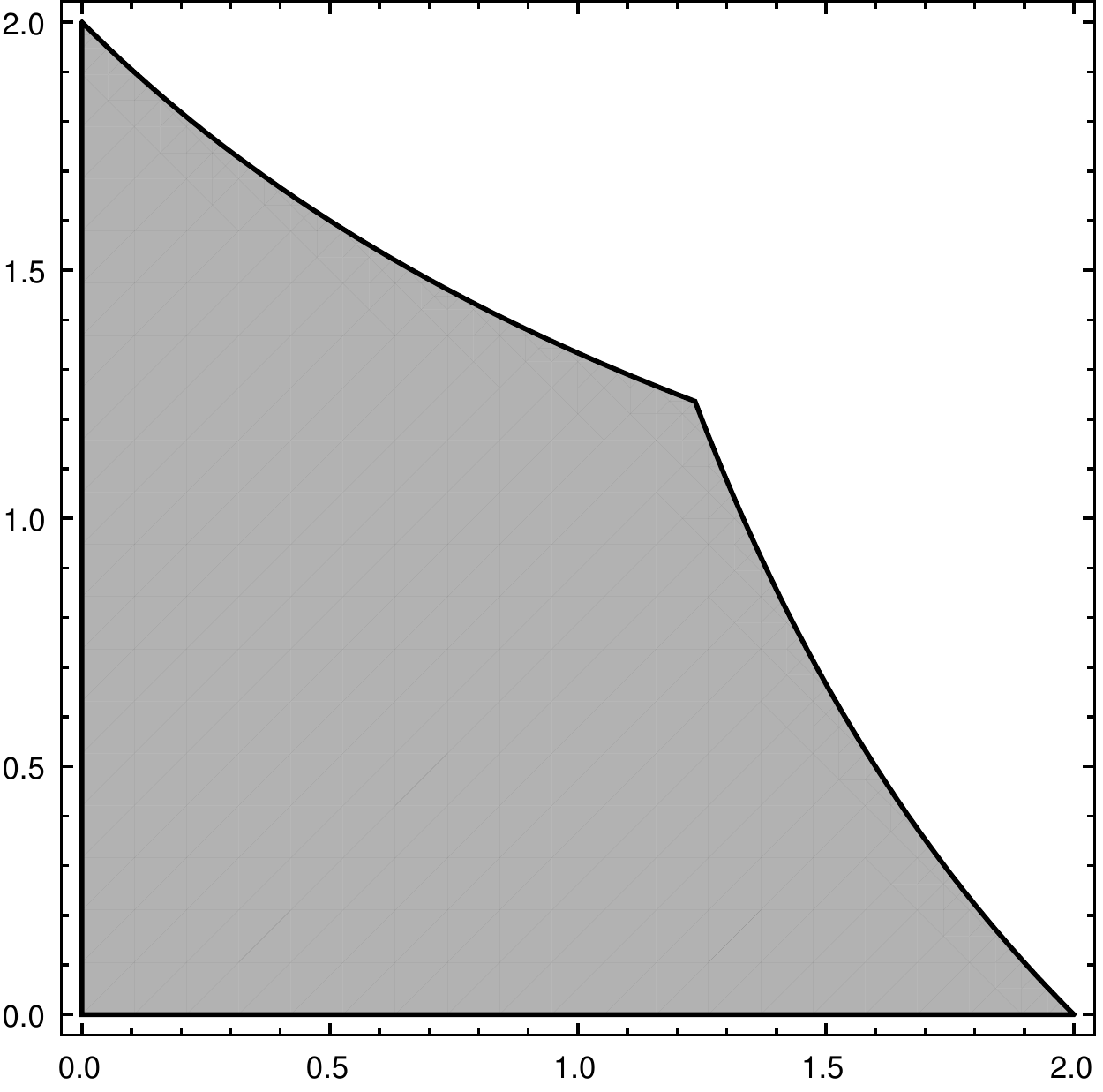}
	\begin{picture}(0,0)%
	    \put(-91,-10){$|\eta|$}%
	    \put(-188,86){$|\zeta|$}%
	\end{picture}%
    \end{center}
    \caption{The domain $\mathcal{D}$ in the $(|\eta|,|\zeta|)$-plane.}\label{ngoccia2}
\end{figure}
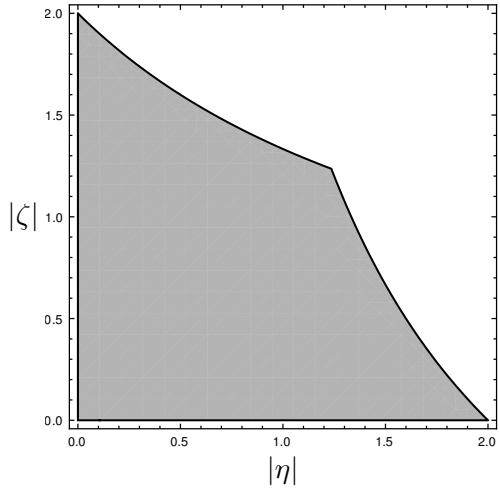
Notice that $\mathcal{H}=0$ only when $\eta=\zeta=0$; it holds $\mathcal{H}(|\eta|,0)=|\eta|$ and $\mathcal{H}(0,|\zeta|)=|\zeta|$.
Moreover, we have that $\mathcal{H}(|\eta|,|\zeta|)$ tends to $+\infty$ when $(|\eta|,|\zeta|)$ tends to the curved edges of $\mathcal{D}$.

Following the notation in Figure~\ref{fig:0}, we set
\begin{equation*}
\mathcal{L}=\{(x,t):\, x<\xa\}\,,\quad \mathcal{M}=\{(x,t): \,\xa<x<\xb\}\,,\quad \mathcal{R}=\{(x,t):\,x>\xb\}\,.
\end{equation*}
We denote $p_o(x) = p\left(v_o(x), \lambda_o(x) \right)$ and $\tv (f,g)= \tv f +\tv g$, for any 
$f=f(x),g=g(x)$.


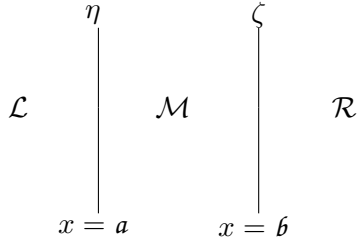
\begin{figure}[htbp]
\begin{picture}(100,100)(-80,-5)
\setlength{\unitlength}{1pt}

\put(170,0){
\put(0,0){\line(0,1){40}} \put(0,75){\makebox(0,0){$\zeta$}}
\put(-60,0){\line(0,1){40}} \put(-62,75){\makebox(0,0){$\eta$}}
\put(-62,-5){\makebox(0,0){$x=\xa$}}\put(-2,-5){\makebox(0,0){$x=\xb$}}
\put(0,40){\line(0,1){30}} 
\put(-60,40){\line(0,1){30}} 
\put(-32,40){\makebox(0,0){$\mathcal{M}$}}
\put(32,40){\makebox(0,0){$\mathcal{R}$}}
\put(-90,40){\makebox(0,0){$\mathcal{L}$}}
}

\end{picture}

\caption{\label{fig:0}{The regions $\mathcal{L}$, $\mathcal{M}$, $\mathcal{R}$ in the $(x,t)$-plane.}}
\end{figure}


\begin{theorem}\label{thm:main}
Assume \eqref{eq:pressure} and consider initial data \eqref{init-data} with $v_o(x)\ge \underline{v}>0$,
for some constant $\underline{v}$. Assume also \eqref{eq:case} and \eqref{stab}.
There exists a strictly decreasing function $\mathcal{K}$ defined for $r>0$, with 
$$
\lim_{r\to 0+} \mathcal{K}(r) = +\infty\,, \qquad \lim_{r\to +\infty}\mathcal{K}(r)=0\,,
$$
such that if it holds
\begin{equation}\label{hyp2}
\tv_{x<\xa}\left(\log(p_o),\frac{u_o}{a_\ell}\right) + \frac{1}{1+\mathcal{H}(|\eta|,|\zeta|)}\tv_{\xa<x<\xb}\left(\log(p_o),\frac{u_o}{a_m}\right) + \tv_{x>\xb}\left(\log(p_o),\frac{u_o}{a_r}\right)<\mathcal{K}\left(\mathcal{H}(|\eta|,|\zeta|)\right),
\end{equation}
then the Cauchy problem \eqref{eq:system}, \eqref{init-data} has a
weak entropic solution $(v,u,\lambda)$ defined for
$t\in\left[0,+\infty\right[$. If $\eta=\zeta=0$ the same conclusion holds with $\mathcal{K}\left(\mathcal{H}(|\eta|,|\zeta|)\right)$ replaced by $+\infty$ in \eqref{hyp2}.

Moreover, the solution is valued in a compact set and $(v(\cdot,t),u(\cdot,t))\in L^\infty([0,\infty[; \BV(\reali))$.
\end{theorem}

The definition of the threshold function $\mathcal{K}$ is given in  \eqref{eq:K-explicit} and is the same as \cite[(6.20)]{ABCD}.
Hypothesis \eqref{hyp2} can be interpreted as follows: the larger $|\eta|,|\zeta|$ can be taken, the smaller the total variation of $p_o,u_o$ must be;
vice versa, the smaller are $|\eta|,|\zeta|$, the larger can be the total variation of $p_o,u_o$.

Consider, now, the case when one of the two phase-waves tends to zero, say $|\eta|\to 0$.
Then $\mathcal{H}(|\eta|,|\zeta|)\to\mathcal{H}(0,|\zeta|)=|\zeta|$ and \eqref{hyp2} becomes formally
\begin{equation}\label{eq:hyp2improved}
\tv_{x<\xa}\left(\log(p_o),\frac{u_o}{a_\ell}\right)+\frac{1}{1+|\zeta|}\tv_{\xa<x<\xb}\left(\log(p_o),\frac{u_o}{a_m}\right)
+ \tv_{x>\xb}\left(\log(p_o),\frac{u_o}{a_r}\right)<\mathcal{K}(|\zeta|)\,,
\end{equation}
which improves \cite[(2.3)]{ABCD} by allowing to take larger total variation of the data for  $x\in\,]\xa,\xb[$. Indeed, when $|\eta|=0$,
we will see in Remark~\ref{rem:hyp2} that hypothesis \eqref{eq:hyp2improved} can be improved by
\begin{equation}\label{eq:hyp2improved2}
\frac{1}{1+|\zeta|}\tv_{x<\xb}\left(\log(p_o),\frac{u_o}{a_m}\right) + \tv_{x>\xb}\left(\log(p_o),\frac{u_o}{a_r}\right)<\mathcal{K}(|\zeta|)\,,
\end{equation}
by which the total variation can be taken larger on the entire interval $]-\infty,\xb[$.

Theorem~\ref{thm:main} improves also the main result in \cite{amadori-corli-siam} when restricted to the case of two contact discontinuities, not only because $\mathcal{K}$ is sharper than $H$ of \cite[Theorem 3.1]{amadori-corli-source} (see Section~\ref{subsec:proof} below), but also because the total variation of the initial data (thanks to the coefficient of the middle term in \eqref{hyp2}) can be larger in $\mathcal{M}$. Recall that $\mathcal{M}$ is the more liquid region if $a(\lambda)$ is increasing. The asymmetrical character of \eqref{hyp2} is due to the particular choice of the decreasing functional $F$ used to estimate the total variation of the approximate solutions, see Section~\ref{sec:interaction}.

We also notice that a slight improvement of condition \eqref{hyp2} in Theorem~\ref{thm:main} would follow from the use of the Riemann coordinates,
see Remark~\ref{rem:RI}.

We conclude this section by extracting some more information from \eqref{hyp2}; with this aim we introduce the sub-level sets of $\mathcal{H}$,
\[
\mathcal{D}_c=\left\{\left(|\eta|,|\zeta|\right)\in\mathcal{D}:\quad  \mathcal{H}\left(|\eta|,|\zeta|\right)<c\right\}\,,\qquad c>0\,, 
\]
see Figure \ref{fig:HH}. Since $\mathcal{K}$ is decreasing, then for every $(|\eta|,|\zeta|)\in \mathcal{D}_c$ condition \eqref{hyp2} holds if
\begin{equation*}
\tv_{x<\xa}\left(\log(p_o),\frac{u_o}{a_\ell}\right) + \tv_{\xa<x<\xb}\left(\log(p_o),\frac{u_o}{a_m}\right) + \tv_{x>\xb}\left(\log(p_o),\frac{u_o}{a_r}\right)< \mathcal{K}(c)\,.
\end{equation*}
In particular, we have $\mathcal{K}(2)=2\log(2+\sqrt 3)/3$ and the domain $\mathcal{D}_2$ includes the segments $[0,2[$
on each axis. Therefore, for $\eta=0$ or $\zeta=0$ we recover a slightly better condition than \cite[(2.5)]{ABCD}. We notice that the $2$-level set of $\mathcal{H}$ has a particular simple expression: it is the graph of the function
$\zeta(|\eta|) = 2(2-|\eta|)/(2+|\eta|)$.


\begin{figure}[htbp]
   \begin{center}
       \includegraphics[width=8cm]{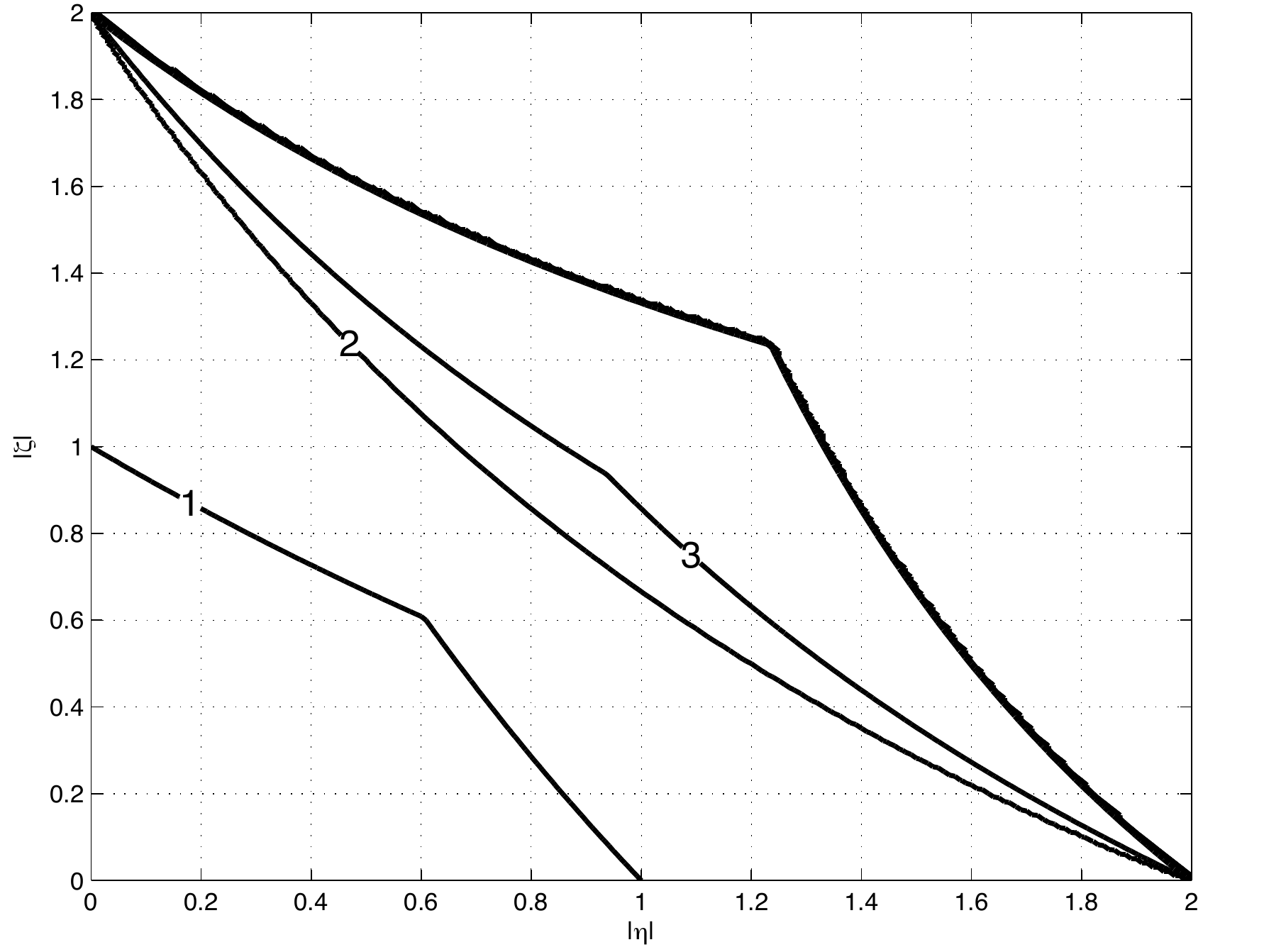}
   \end{center}
   \vspace{-3ex}
   \caption{\label{fig:HH}
   Sets of level $c$ of the function $\mathcal{H}$: cases $c=1,2,3$.}
\end{figure}

\section{The Riemann Problem}\label{sec:Riemann}
\setcounter{equation}{0}

In this section we collect some basic facts about system \eqref{eq:system};  we refer to \cite{ABCD,AmCo06-Proceed-Lyon,amadori-corli-siam}
for more details and to \cite{Bressanbook, Dafermos} for generalities on Riemann problems. As anticipated in the Introduction, in addition to the usual
Lax waves used in the theory of conservation laws, we introduce suitable composite waves which sum up the effects of each contact discontinuity and of
certain reflected waves. Finally, we present two Riemann solvers that use such composite waves.

For $i=1,3$, the $i$-th right shock-rarefaction curves $\Phi_i$ through the point $U_o = (v_o,u_o,\lambda_o)\in\Omega$ for \eqref{eq:system}
are as in \cite{amadori-corli-siam}
\begin{equation}
v \mapsto \Phi_i(\eps_i)(U_o)=\left(v,u_o + 2a(\lambda_o) h(\eps_i),\lambda_o\right)\,,\qquad v>0\,, \ i=1,3\,,
\label{eq:lax13}
\end{equation}
where the strength $\eps_i$ of an $i$-wave is defined as
\begin{equation}\label{eq:strengths}
\eps_1=\frac{1}{2}\log\left(\frac{v}{v_o}\right)=\frac{1}{2}\log\left(\frac{p}{p_o}\right)\,,
\qquad \eps_3=\frac{1}{2}\log\left(\frac{v_o}{v}\right)=\frac{1}{2}\log\left(\frac{p_o}{p}\right)
\end{equation}
and the function $h$ is defined by
\begin{equation}\label{h}
h(\eps)= \begin{cases}
\eps& \mbox{ if } \eps \ge 0\,,\\
\sinh \eps& \mbox{ if } \eps < 0\,.
\end{cases}
\end{equation}
Rarefaction waves have positive strengths and shock waves have negative strengths. The $i$-th integral curve through $U_o\in \Omega$ is denoted by
$I_i(\eps)(U_o)$, for $\eps\in\reali$ and $i=1,3$; two states $U$ and $I_i(\eps)(U)$ are connected by an $i$-rarefaction wave iff $\eps>0$. The wave
curve corresponding to the second characteristic field through $U_o\in\Omega$ is given by
\begin{equation*}
\lambda\mapsto\left(v_o\ds\frac{a^2(\lambda)}{a^2(\lambda_o)},
u_o,\lambda\right)\,,\qquad \lambda\in[0,1],
\end{equation*}
and the strength of a $2$-wave is
\begin{equation*}
\eps_2 = 2\, \frac{a(\lambda)-a(\lambda_o)}{a(\lambda)+a(\lambda_o)}\,.
\end{equation*}

\bigskip
For starters, we prove a result similar to \cite[Proposition 3.2]{amadori-corli-siam}. For $\lambda_\pm\in[0,1]$, we use the notation
$a_{\pm}=a(\lambda_{\pm}),\,p_{\pm}=p(v_{\pm},\lambda_{\pm})$.

\begin{proposition}\label{prop:RP}
Fix two functions $\theta_1,\theta_3$ that can be either the identity $\Id$ or the function $h$ defined in \eqref{h}.
For any pair of states $U_-=(v_-,u_-,\lambda_-),U_+=(v_+,u_+,\lambda_+)\in\Omega$, there exist unique $\eps_1,\eps_3\in \reali$ such that:
\begin{equation}\label{impo}
\eps_3-\eps_1=\frac{1}{2}\log\left(\frac{p_+}{p_-}\right)\,, \qquad a_-\theta_1(\eps_1)+a_+\theta_3(\eps_3)=\frac{u_+-u_-}{2}\,.
\end{equation}
\end{proposition}

\begin{proof}
Let us call $\log(p_+/p_-)/2=:A$ and $(u_+-u_-)/2=:B$, since they are two constant quantities once we fixed $U_-$ and $U_+$.
Thus, we have four possible cases to examine:
\begin{align}
&\left\{
\begin{array}{ll}
\eps_3-\eps_1=A\,,\\
a_- h(\eps_1)+a_+ h(\eps_3)=B\,,
\end{array}
\right.
& &\left\{
\begin{array}{ll}
\eps_3-\eps_1=A\,,\\
a_- \eps_1+a_+ \eps_3=B\,,
\end{array}
\right.\label{old}\\
&\left\{
\begin{array}{ll}
\eps_3-\eps_1=A\,,\\
a_- h(\eps_1)+a_+ \eps_3=B\,,
\end{array}
\right.
& &\left\{
\begin{array}{ll}
\eps_3-\eps_1=A\,,\\
a_- \eps_1+a_+ h(\eps_3)=B\,.
\end{array}
\right.\label{new}
\end{align}
System \eqref{old}$_{1}$ ($\theta_1=\theta_3=h$) has already been solved in \cite[Proposition 3.2]{amadori-corli-siam},
while system \eqref{old}$_{2}$ is linear. As for \eqref{new}, it suffices to study just one of the two systems, for example \eqref{new}$_{1}$ (the other one
is analogous). In this case, setting $k=a_+/a_-$, it holds
$h(\eps_1)+k\eps_1=B/(a_{-})-kA$. Thus, if $G(x):=kx+h(x)$, we have
$G(\eps_1)=B/(a_{-})-kA$.
Since $G$ is invertible and onto, this gives $\eps_1=G^{-1}(B/(a_{-})-kA)$.
\end{proof}

\begin{remark}\label{rem:preRS}
Notice that only system \eqref{old}$_1$ always gives an actual Lax solution to the Riemann problem for \eqref{eq:system} with initial data
\begin{equation}\label{eq:RP}
U(x,0)=\left\{
\begin{array}{ll}
U_- & \hbox{ if }x<0\,,
\\
U_+ & \hbox{ if }x>0\,,
\end{array}
\right.
\end{equation}
as the juxtaposition of a $1$-wave of strength $\eps_1$, a $2$-wave $\delta=2(a_+-a_-)/(a_++a_-)$ and a $3$-wave of strength $\eps_3$, see \cite[Proposition 3.2]{amadori-corli-siam}. In general, this is not true for the other three cases.
\end{remark}

When solving an interaction with a $2$-wave $\delta$, we sometimes make use of a Riemann solver that attaches certain reflected waves to $\delta$; the outcome is a stationary composite wave, which is made of the composition of a wave related to an integral curve for the first characteristic field, the $2$-wave $\delta$, and, finally, a wave related to an integral curve for the third characteristic field.

We use the symbols `$L$' to refer to the Lax curves $\Phi_i$ and `$I$' to refer to the Integral curves $I_i$, $i=1,3$.  
Then, Proposition~\ref{prop:RP} allows us to give the following important definition.

\begin{definition}[Pre-Riemann solver]\label{def:preRS}
For any choice of $\theta_1,\theta_3$ as in Proposition~\ref{prop:RP}, the \emph{Pre-Riemann solver} $R_{\theta_1\!\theta_3}:\Omega\times\Omega\to \reali\times ]-2,2[\times \reali$ is the map defined by
\begin{equation}\label{eq:Rsolver}
R_{\theta_1\!\theta_3}(U_-,U_+)=
(\eps_1,\delta,\eps_3)\,,
\end{equation}
where $\eps_1,\eps_3$ are as in \eqref{impo} and $\delta=2(a_+-a_-)/(a_++a_-)$. The two subscripts in $\theta_1,\theta_3$ stand for the choice of $1,3$-wave curves ($L$ or $I$) along which $\eps_1,\eps_3$ are taken. More precisely, it holds $\theta_i=h$ for $\eps_i$ along Lax curves, while $\theta_i=\Id$ for $\eps_i$ along integral curves. Then, by Proposition~\ref{prop:RP} we get four Pre-Riemann solvers that we denote by $\RLL$, $\RII$, $\RLI$ and $\RIL$, respectively.
\end{definition}

Notice that $\RLL$ is an actual Riemann solver by Remark~\ref{rem:preRS} and $\RIL$, $\RLI$ are used in connection with the Simplified Riemann Solver, see Proposition~\ref{prop:accsimpl}. We do not assign any speed to $\eps_1,\eps_3$ when they are taken along integral curves $I$; indeed, these waves shall be sticked to the phase wave and can be thought as being stationary. In particular, $\RII$ is used to define composite waves as in the following definition.

\begin{definition}[Composite wave]\label{def:compositewave}
A composite wave $\delta_0=(\delta_0^1,\delta,\delta_0^3)$ associated to a $2$-wave $\delta$ and connecting two states
$U_-=(v_-,u_-,\lambda_-)$ and $U_+=(v_+,u_+,\lambda_+)$ of $\Omega$, with $\lambda_-\neq\lambda_+$, is the wave with
\emph{zero speed} defined by $\delta_0=\RII(U_-,U_+)$. We write $|\delta_0|=|\delta_0^1|+|\delta_0^3|$.
\end{definition}

Notice that $\delta_0$ reduces to a $2$-wave as long as $\delta_0^1=\delta_0^3=0$.
We denote by $\eta_0$ and $\zeta_0$ the two composite waves associated to $\eta$ and $\zeta$, respectively; see Figure~\ref{fig:1}.


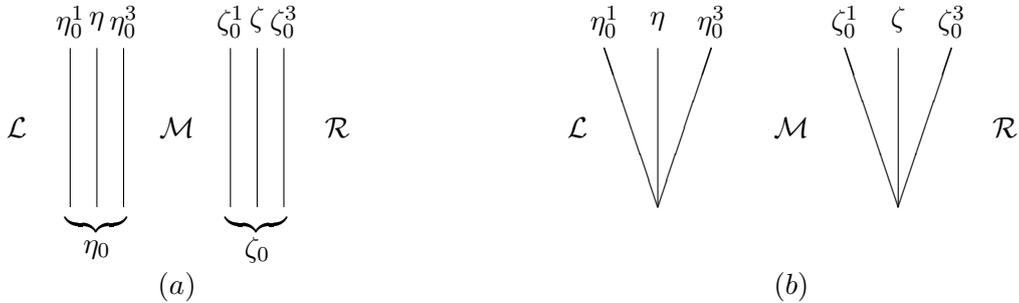
\begin{figure}[htbp]
\begin{picture}(120,120)(-120,10)
\setlength{\unitlength}{1pt}

\put(-70,0){
\put(10,40){\line(0,1){60}}
\put(20,40){\line(0,1){60}}
\put(30,40){\line(0,1){60}}
\put(10,110){\makebox(0,0){$\eta_0^1$}}
\put(20,110){\makebox(0,0){$\eta$}}
\put(30,110){\makebox(0,0){$\eta_0^3$}}

\put(8,37){$\underbrace{\phantom{aaaa}}$}
\put(20,24){\makebox(0,0){$\eta_0$}}

\put(70,40){\line(0,1){60}}
\put(80,40){\line(0,1){60}}
\put(90,40){\line(0,1){60}}
\put(70,110){\makebox(0,0){$\zeta_0^1$}}
\put(80,111){\makebox(0,0){$\zeta$}}
\put(90,110){\makebox(0,0){$\zeta_0^3$}}

\put(68,37){$\underbrace{\phantom{aaaa}}$}
\put(80,24){\makebox(0,0){$\zeta_0$}}

\put(-10,70){\makebox(0,0){$\mathcal{L}$}}
\put(50,70){\makebox(0,0){$\mathcal{M}$}}
\put(110,70){\makebox(0,0){$\mathcal{R}$}}

\put(50,10){\makebox(0,0){$(a)$}}
}

\put(200,0){
\put(-40,40){\line(1,3){20}} \put(-40,40){\line(0,1){60}}
\put(-40,40){\line(-1,3){20}}
\put(-60,110){\makebox(0,0){$\eta_0^1$}}
\put(-40,110){\makebox(0,0){$\eta$}}
\put(-20,110){\makebox(0,0){$\eta_0^3$}}

\put(50,40){\line(1,3){20}} \put(50,40){\line(0,1){60}}
\put(50,40){\line(-1,3){20}}
\put(30,110){\makebox(0,0){$\zeta_0^1$}}
\put(50,111){\makebox(0,0){$\zeta$}}
\put(70,110){\makebox(0,0){$\zeta_0^3$}}

\put(-70,70){\makebox(0,0){$\mathcal{L}$}}
\put(10,70){\makebox(0,0){$\mathcal{M}$}}
\put(90,70){\makebox(0,0){$\mathcal{R}$}}

\put(10,10){\makebox(0,0){$(b)$}}
}

\end{picture}
\caption{\label{fig:1}{The composite waves in the $(x,t)$ plane: in $(a)$ $\eta_0$ and $\zeta_0$ are drawn as three parallel close lines, while $(b)$ is the auxiliary picture that  is used to determine the states in the interactions, see Figure~\ref{fig:interactions}.}}
\end{figure}


\noindent Remark that in Figure~\ref{fig:1} $(b)$ the $\eta_0^i,\zeta_0^i$ components, $i=1,3$, may be non-entropic waves: they are depicted as fronts with negative speed ($i=1$) and positive speed ($i=3$) in order to easily understand how to handle the interactions.

In this way, we are left to deal with waves of family $1,3$ and two distinct composite waves belonging to a fictitious $0$-family.
Notice also that, once we fix $U_-$, the set of states $U_+$ that can be connected to $U_-$ by a composite wave does not describe a curve in the
$(v,u)$ plane, but the whole half-plane $v>0$.

Before proceeding with the detailed description of the two new Riemann solvers, we insert here the following useful lemma.
Now and then we will make an inappropriate use of the term `waves' to indicate both actual physical waves (i.e.\ connecting states that lie on a Lax curve)
and not (i.e.\ when referring to states that lie on a general integral curve or on a combination of Lax curves and integral curves).

\begin{lemma}[Commutation of $i$-waves]\label{lem:commuta}
Let $i=1,3$ and $\alpha_i,\beta_i\in\reali$. If two states $U_-,U_+\in \Omega$ in the same phase ($\lambda_-=\lambda_+$) are connected by an
$i$-wave of strength $\alpha_i$ followed by an $i$-wave of strength $\beta_i$, then they can be connected also by an $i$-wave $\beta_i$ followed by
an $i$-wave $\alpha_i$.
\end{lemma}
\begin{proof}
Assume $i=3$ (the other case is analogous) and fix $\theta_3^{\alpha_3},\theta_3^{\beta_3}$ to be either $h$ or $\Id$, see Figure~\ref{fig:commutation}. 
If $U^{*}=(v^{*},u^{*},\lambda^{*})$ is the final state reached starting from $U_-$ and moving first along $\beta_3$ and then along $\alpha_3$, 
then trivially it holds $\lambda^{*}=\lambda_+$ and
$$
v^{*}=v_-\exp(-2\beta_3-2\alpha_3)=v_+\,, \qquad
u^{*}=u_-+2a_-\left(\theta_3^{\beta_3}(\beta_3)+\theta_3^{\alpha_3}(\alpha_3)\right)=u_+\,,
$$
that means $U^{*}=U_+$.
\end{proof}


\begin{figure}[htbp]
   \begin{center}
       \includegraphics{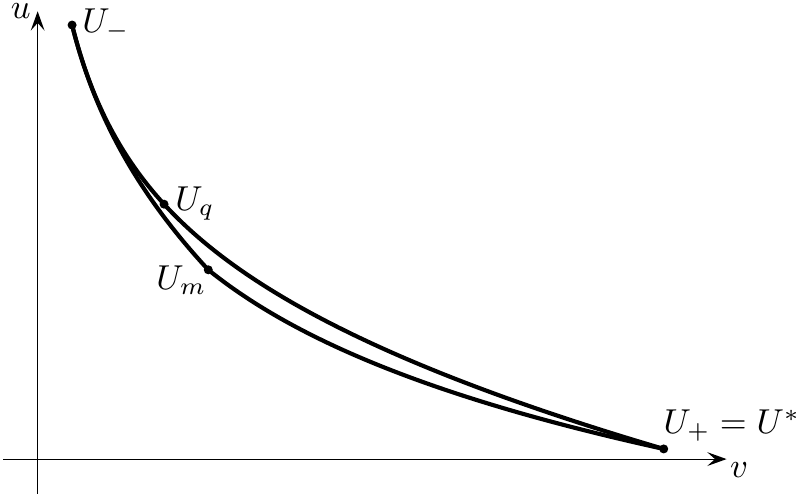}
   \end{center}
   \vspace{-3ex}
   \caption{\label{fig:commutation}
   The commutation of $i$-waves: case $i=3$, $\alpha_3,\beta_{3}<0$,
   $\theta_3^{\alpha_3}=h$ and $\theta_3^{\beta_3}=\Id$.
   Here $U_m$ and $U_{q}$ are the states connected to $U_{-}$ along the $3$-Lax
   curve by $\alpha_3$ and, respectively, along the $3$-integral curve by $\beta_3$.}
\end{figure}


\begin{remark}
When $\theta_i^{\alpha_i}=\theta_i^{\beta_i}=h$, Lemma~\ref{lem:commuta} is a consequence of the invariance by translation of Lax curves
for the $p$-system with $\gamma=1$; see \cite{Nishida68}.
\end{remark}

Now, we are ready to describe the two Riemann solvers that will be needed in case of interactions with $\eta_0$ and $\zeta_0$ at positive times:
we use an \emph{Accurate solver} when the interacting wave has size bigger than a threshold $\rho$ to be determined and a \emph{Simplified solver}
otherwise.

We denote by $\delta_i$ (and $\eps_i$) the interacting waves (the waves produced by the interaction, respectively), for $i=1,3$; note that, taking for
simplicity $\delta$ to be equal either to $\eta$ or to $\zeta$, we use the same notation $\delta_0=(\delta_0^1,\delta,\delta_0^3)$
(and $\eps_0=(\eps_0^1,\delta,\eps_0^3)$) to denote both $\eta_0$ and $\zeta_0$ as interacting waves (and as outgoing waves, respectively).

\begin{proposition}\label{prop:accsimpl}
Let $i=1,3$ and consider the interaction at a time $t>0$ of a composite wave $\delta_0=(\delta_0^1,\delta,\delta_0^3)$ with an $i$-wave of strength $\delta_i$; we refer to  Figure~\ref{fig:interactions}. Then, the emerging Riemann problem with initial states $U_-,U_+$ can be solved by means of $R_{\theta_1\!\theta_3}$ in one of the two following ways. Denote $\widetilde{U}_-=I_1(\delta_0^1)(U_-)$ and
$\widetilde{U}_+=I_3(-\delta_0^3)(U_+)$.
\begin{enumerate}
\item \emph{Accurate Riemann solver}. The solution is formed by waves $\eps_1,\eps_0,\eps_3$, where $(\eps_1,\delta,\eps_3)=\RLL(\widetilde{U}_-,\widetilde{U}_+)$ and $\eps_0=\delta_0$.
\item \emph{Simplified Riemann solver}. We distinguish case $i=1$ and $i=3$: 
\begin{itemize}
\item[i)] for $i=1$, the solution is formed by waves $\eps_1,\eps_0$ such that $(\eps_1,\delta,\eps_3)=\RLI(\widetilde{U}_-,\widetilde{U}_+)$ and $\eps_0=(\delta_0^1,\delta,\delta_0^3+\eps_3)$;
\item[ii)] for $i=3$, the solution is formed by waves $\eps_0,\eps_3$ such that $(\eps_1,\delta,\eps_3)=\RIL(\widetilde{U}_-,\widetilde{U}_+)$ and $\eps_0=(\delta_0^1+\eps_1,\delta,\delta_0^3)$.
\end{itemize}
\end{enumerate}

In general, it holds $\theta_i=h$ in all cases; for any $\theta_j$, $j=1,3$, $j\neq i$, chosen between $\Id$ and $h$, the following relations are verified:
\begin{equation}\label{eq:rels}
\eps_3-\eps_1=\begin{cases}
-\delta_1 &\quad \text{if $i=1$}\,,\\
\delta_3 &\quad \text{if $i=3$}\,,
\end{cases}
\qquad a_- \theta_1(\eps_1)+a_+ \theta_3(\eps_3)=\begin{cases}
a_+ \theta_1(\delta_1) &\quad \text{if $i=1$}\,,\\
a_-\theta_3(\delta_3) &\quad \text{if $i=3$}\,.
\end{cases}
\end{equation}

Moreover, in all cases the signs of $\eps_1,\eps_3$ satisfy:
\begin{equation}\label{eq:sign}
\sgn\,\eps_i = \sgn\,\delta_i\,,
\qquad
\sgn\,\eps_j  = \begin{cases}
\sgn\,\delta\cdot \sgn\,\delta_i &\hbox{ if } i=1\,,
\\
-\sgn\,\delta\cdot \sgn\,\delta_i &\hbox{ if } i=3\,.
\end{cases}
\end{equation}
\end{proposition}

\begin{proof}
In the interaction of an $i$-wave $\delta_i$ with a composite wave $\delta_0$, we look at the interaction of $\delta_i$ with the $\delta$ component of $\delta_0$: indeed, $\delta_i$ crosses $\delta_0^j $, $j=1,3$, $j\neq i$, without changing strength by \cite[Lemma 5.4]{ABCD}. Then, we solve the Riemann problem with initial states $\widetilde{U}_-,\widetilde{U}_+$ by means of $R_{\theta_1\!\theta_3}$, with $\theta_1,\theta_3$ either $\Id$ or $h$. We proceed as follows.

\begin{enumerate}
\item \emph{Accurate Riemann solver}. After computing $\RLL(\widetilde{U}_-,\widetilde{U}_+)=(\eps_1,\delta,\eps_3)$, we let $\eps_1$ and $\eps_3$ commute with $\delta_0^1$ and $\delta_0^3$ respectively, in the sense of Lemma~\ref{lem:commuta}. In this way, they are free to propagate as outgoing waves of family $1,3$; see Figure~\ref{fig:interactions} $(a),(b)$ for a picture of case $i=3$. Then, the resulting composite wave connects $U_p$ to $U_q$, where $U_p=\Phi_1(\eps_1)(U_-)$ and $U_q=\Phi_3(-\eps_3)(U_+)$. Hence, $\eps_0=\RII(U_p,U_q)=(\delta_0^1,\delta,\delta_0^3)=\delta_0$.
\item \emph{Simplified Riemann solver}. We have to distinguish between case $i=1$ and $i=3$. Once the triple $(\eps_1,\delta,\eps_3)$ has been determined by $\RLI$ or $\RIL$, the idea is to `project' the reflected wave along the associated integral curve; see Figure~\ref{fig:interactions} $(c),(d)$ for a picture of case $i=3$.
\begin{itemize}
\item[i)] For $i=1$, we compute $\RLI(\widetilde{U}_-,\widetilde{U}_+)=(\eps_1,\delta,\eps_3)$ and let $\eps_1$ commute with $\delta_0^1$ by Lemma~\ref{lem:commuta}. The outgoing composite wave connects $U_p$ to $U_+$, where $I_1(\delta_0^1)(U_p)=\Phi_1(\eps_1)(\widetilde{U}_-)$
and $U_+=I_3(\delta_0^3+\eps_3)\circ\Phi_2(\delta)\circ\Phi_1(\eps_1)(\widetilde{U}_-)$. Hence, $\eps_0=\RII(U_p,U_+)=(\delta_0^1,\delta,\delta_0^3+\eps_3)$.
\item[ii)] For $i=3$, we compute $\RIL(\widetilde{U}_-,\widetilde{U}_+)=(\eps_1,\delta,\eps_3)$ and let $\eps_3$ commute with $\delta_0^3$ by Lemma~\ref{lem:commuta}. The outgoing composite wave connects $U_-$ to $U_q$,  
where $U_q=I_3(\delta_0^3)\circ\Phi_2(\delta)\circ I_1(\delta_0^1+\eps_1)(U_-)$. Hence, $\eps_0=\RII(U_-,U_q)=(\delta_0^1+\eps_1,\delta,\delta_0^3)$.
\end{itemize}
\end{enumerate}

To prove \eqref{eq:rels}, notice that for $i=1$ we use $\RLL$ or $\RLI$ (i.e.\ $\theta_1=h$), while for $i=3$ we use $\RLL$ or $\RIL$ (i.e $\theta_3=h$). Hence, \eqref{eq:rels}$_2$ is equivalent to
$$
a_- \theta_1(\eps_1)+a_+ \theta_3(\eps_3)=\begin{cases}
a_+ h(\delta_1) &\quad \text{if $i=1$}\,,\\
a_-h(\delta_3) &\quad \text{if $i=3$}\,.
\end{cases}
$$
By \eqref{eq:lax13} and \eqref{eq:strengths} we have that for $i=1,3$
$$
\frac{1}{2}\log \left(\frac{\widetilde{p}_+}{\widetilde{p}_-}\right) = \begin{cases} -\delta_1 &\quad \text{if $i=1$}\,,\\
\delta_3 &\quad \text{if $i=3$}\,,
\end{cases} \qquad   \frac{\widetilde{u}_+-\widetilde{u}_-}{2}=\begin{cases}
											\, a_+h(\delta_1) &\quad \text{if $i=1$}\,,\\
											\, a_-h(\delta_3) &\quad \text{if $i=3$}\,.
										  \end{cases}
$$
Now, by Proposition~\ref{prop:RP} it suffices to notice that
$$
\eps_3-\eps_1=\frac{1}{2}\log \left(\frac{\widetilde{p}_+}{\widetilde{p}_-}\right)\,, \qquad  a_-\theta_1(\eps_1)+a_+\theta_3(\eps_3)=\frac{\widetilde{u}_+-\widetilde{u}_-}{2}\,.
$$
Hence, \eqref{eq:rels} holds.


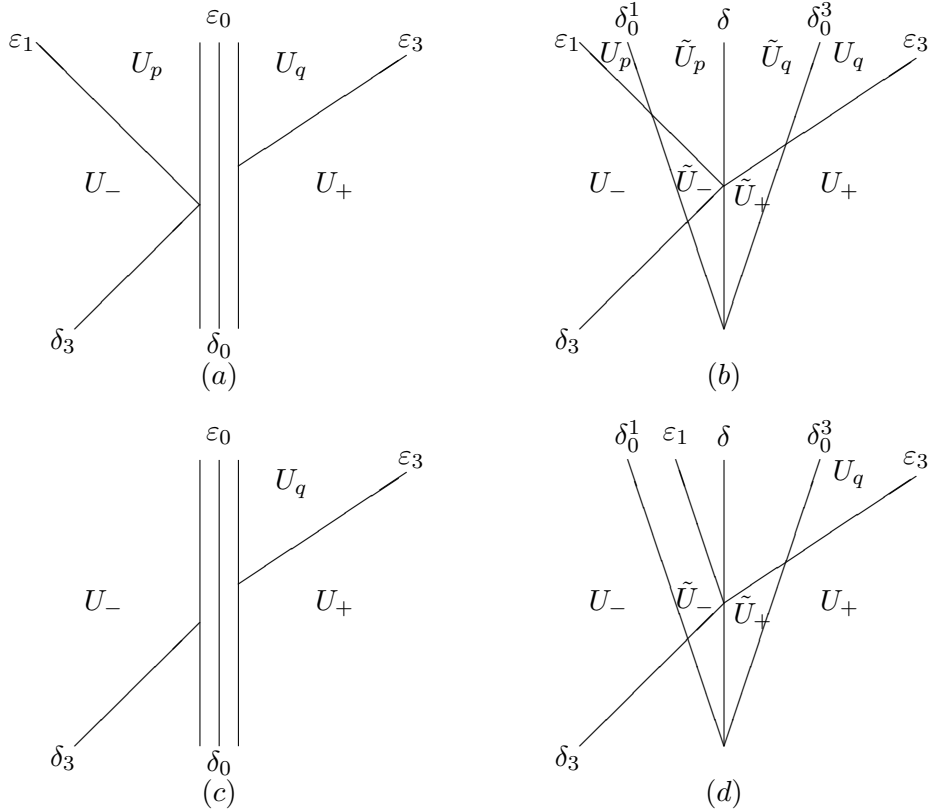
\begin{figure}[htbp]
\begin{picture}(100,280)(-100,-160)
\setlength{\unitlength}{0.9pt}

\put(40,0){

\put(0,0){\line(0,1){120}}
\put(8,0){\line(0,1){120}}
\put(-8,0){\line(0,1){120}}
\put(0,130){\makebox(0,0){$\eps_0$}}
\put(0,-7){\makebox(0,0){$\delta_0$}}
\put(-60,0){\line(1,1){52}}
\put(-65,-5){\makebox(0,0){$\delta_3$}}
\put(8,68){\line(3,2){70}}
\put(80,120){\makebox(0,0){$\eps_3$}}
\put(-8,52){\line(-1,1){68}}
\put(-82,120){\makebox(0,0){$\eps_1$}}
\put(-40,60){\makebox(0,0)[r]{$U_-$}}
\put(40,60){\makebox(0,0)[l]{$U_+$}}
\put(-30,110){\makebox(0,0){$U_p$}}
\put(30,110){\makebox(0,0){$U_q$}}
\put(0,-20){\makebox(0,0){$(a)$}}
}

\put(250,0){
\put(0,0){\line(0,1){120}}
\put(0,0){\line(-1,3){40}}
\put(0,0){\line(1,3){40}}
\put(0,130){\makebox(0,0){$\delta$}}
\put(-40,130){\makebox(0,0){$\delta_0^1$}}
\put(40,130){\makebox(0,0){$\delta_0^3$}}
\put(-60,0){\line(1,1){60}}
\put(-65,-5){\makebox(0,0){$\delta_3$}}
\put(0,60){\line(3,2){80}}
\put(80,120){\makebox(0,0){$\eps_3$}}
\put(0,60){\line(-1,1){60}}
\put(-65,120){\makebox(0,0){$\eps_1$}}
\put(-40,60){\makebox(0,0)[r]{$U_-$}}
\put(-4,62){\makebox(0,0)[r]{$\tilde{U}_-$}}
\put(4,56){\makebox(0,0)[l]{$\tilde{U}_+$}}
\put(40,60){\makebox(0,0)[l]{$U_+$}}
\put(-45,114){\makebox(0,0){$U_p$}}
\put(-14,114){\makebox(0,0){$\tilde{U}_p$}}
\put(22,114){\makebox(0,0){$\tilde{U}_q$}}
\put(52,114){\makebox(0,0){$U_q$}}
\put(0,-20){\makebox(0,0){$(b)$}}
}

\put(40,-175){
\put(0,0){\line(0,1){120}}
\put(8,0){\line(0,1){120}}
\put(-8,0){\line(0,1){120}}
\put(0,130){\makebox(0,0){$\eps_0$}}
\put(0,-7){\makebox(0,0){$\delta_0$}}
\put(-60,0){\line(1,1){52}}
\put(-65,-5){\makebox(0,0){$\delta_3$}}
\put(8,68){\line(3,2){70}}
\put(80,120){\makebox(0,0){$\eps_3$}}
\put(-40,60){\makebox(0,0)[r]{$U_-$}}
\put(40,60){\makebox(0,0)[l]{$U_+$}}
\put(30,110){\makebox(0,0){$U_q$}}
\put(0,-20){\makebox(0,0){$(c)$}}
}

\put(250,-175){
\put(0,0){\line(0,1){120}}
\put(0,0){\line(-1,3){40}}
\put(0,0){\line(1,3){40}}
\put(0,130){\makebox(0,0){$\delta$}}
\put(-40,130){\makebox(0,0){$\delta_0^1$}}
\put(40,130){\makebox(0,0){$\delta_0^3$}}
\put(-60,0){\line(1,1){60}}
\put(-65,-5){\makebox(0,0){$\delta_3$}}
\put(0,60){\line(3,2){80}}
\put(80,120){\makebox(0,0){$\eps_3$}}
\put(0,60){\line(-1,3){20}}
\put(-20,130){\makebox(0,0){$\eps_1$}}
\put(-40,60){\makebox(0,0)[r]{$U_-$}}
\put(-4,62){\makebox(0,0)[r]{$\tilde{U}_-$}}
\put(4,56){\makebox(0,0)[l]{$\tilde{U}_+$}}
\put(40,60){\makebox(0,0)[l]{$U_+$}}
\put(52,114){\makebox(0,0){$U_q$}}
\put(0,-20){\makebox(0,0){$(d)$}}
}

\end{picture}
\vspace{20pt}
\caption{\label{fig:interactions}{Interaction of a $3$-wave $\delta_3$ with a composite wave $\delta_0$. $(a)$, $(c)$: the actual Riemann solvers, the Accurate case $(a)$ and the Simplified one $(c)$; $(b)$, $(d)$: the auxiliary pictures, the Accurate case $(b)$ and the Simplified one $(d)$.}}
\end{figure}


Finally, we verify the relations on the signs of the outgoing waves $\eps_1,\eps_3$. We prove only case $\delta>0$, since the other one is symmetric by replacing $i=1$ with $i=3$. Notice that for $\RLL$ the results collected in \eqref{eq:sign} have already been proved in \cite{AmCo06-Proceed-Lyon} and we obtain the same interaction patterns of \cite[(5.5)]{ABCD}. \\
When the Simplified solver is used and $i=3$, by \eqref{eq:rels} $\eps_1,\eps_3$ solve
\begin{equation}\label{eq:sys1}
\begin{cases}
\eps_3-\eps_1=\delta_3\,,\\
a_-\eps_1+a_+h(\eps_3)=a_-h(\delta_3)\,.
\end{cases}
\end{equation}
Substituting the expression for $\eps_1$ coming from the first equation of \eqref{eq:sys1} into the second one, we obtain $\eps_3+k h(\eps_3)=\delta_3+h(\delta_3)$, where $k=a_+/a_->1$. Hence, we have that $\sgn\, \eps_3=\sgn\, \delta_3$. Now, take $\delta_3<0$ and assume to use $\RLL$ to solve the Riemann problem at the point of interaction. The corresponding outgoing waves $\eps_1^*,\eps_3^*$ solve
\begin{equation}\label{eq:sys1bis}
\begin{cases}
\eps^*_3-\eps^*_1=\delta_3\,,\\
a_-h(\eps^*_1)+a_+h(\eps_3^*)=a_-h(\delta_3)\,.
\end{cases}
\end{equation}
Since $\eps_1^*>0$, then system \eqref{eq:sys1bis} reduces to \eqref{eq:sys1} and, by uniqueness, its solution coincides precisely with $\eps_1,\eps_3$. Hence, \eqref{eq:sign} is valid. If $\delta_3>0$, instead, we have that $h(\delta_3)=\delta_3$ and $h(\eps_3)=\eps_3$, i.e.\ in this case it holds $\RIL=\RII$. This amounts to solve a linear system in $\eps_1,\eps_3$ and we find $\eps_1=-\delta_3\delta/2$. Hence, $\sgn\,\eps_1=-\sgn\,\delta_3=-\sgn\,\delta\cdot\sgn\,\delta_3$, as wished.\\
When $i=1$, by \eqref{eq:rels} $\eps_1,\eps_3$ solve
\begin{equation}\label{eq:sys2}
\begin{cases}
\eps_3-\eps_1=-\delta_1\,,\\
a_-h(\eps_1)+a_+\eps_3=a_+h(\delta_1)\,.
\end{cases}
\end{equation}
Again, it is easy to prove that $\sgn\,\eps_1=\sgn \,\delta_1$. If $\delta_1>0$, then $\RLI=\RII$ and system \eqref{eq:sys2} is linear. Thus, we get $\eps_3=\delta_1\delta/2$ and $\sgn\,\eps_3=\sgn\,\delta_1=\sgn\,\delta\cdot \sgn\,\delta_1$, as wished. If, instead, $\delta_1<0$, then the second formula in \eqref{eq:sys2} becomes
\begin{equation}
    \label{eq:delta1eq}
    \sinh\eps_1+k\eps_3=k\sinh\delta_1\,,
\end{equation}
where $k=a_+/a_->1$. By substituting the expression for $\eps_1$ obtained from the first equation of \eqref{eq:sys2} in \eqref{eq:delta1eq}, we get $k(\eps_3+\delta_1)+\sinh(\eps_3+\delta_1)=k\left(\sinh\delta_1+\delta_1\right)$. If we call $\Gamma(x):=kx+\sinh x$, then $\Gamma(\eps_3+\delta_1)=k\left(\sinh\delta_1+\delta_1\right)$ and
\begin{equation}
    \label{eq:delta1eq2}
    \Gamma(\eps_3+\delta_1)-\Gamma(\delta_1)=(k-1)\sinh\delta_1\,.
\end{equation}
Since $\Gamma$ is a strictly increasing function and $\delta_1<0$, it  follows $\eps_3<0$, that is $\sgn \,\eps_3=\sgn \, \delta_1=\sgn\,\delta\cdot\sgn\,\delta_1$. Therefore, the proposition is completely proved.
\end{proof}

\begin{remark}\label{rem:simpl}
Differently from \cite{ABCD}, in the Simplified Riemann solver the emerging error is not only on the $u$-component of the $0$-wave and the transmitted $i$-wave $\eps_i$ does not maintain the same strength $\delta_i$ of the incoming one. The latter is a key feature of the solver, that guarantees the decrease of the functional defined in \eqref{F} across any interaction. Indeed, here we take into account the possible appearance of a reflected $j$-wave $\eps_j$ ($j \neq i$) that we attach to $\delta_0$ in place of a standard non-physical wave as in \cite{amadori-corli-siam}: this is possible because the states connected by $\eps_j$ and by the $\delta_0^j$ component lie on the same $j$-integral curve. See Remark~\ref{rem:why} for more details.
\end{remark}


\section{Approximate solutions}\label{sec:app_sol}
\setcounter{equation}{0}

We use Proposition~\ref{prop:accsimpl} to build up the piecewise-constant approximate solutions to (\ref{eq:system}) that are needed for the wave-front tracking scheme \cite{Bressanbook,amadori-corli-siam}. We first approximate the initial data \eqref{init-data}: for any $\nu\in\naturali$ we take a sequence $(v^\nu_o,u^\nu_o)$ of piecewise constant functions with a finite number of jumps such that, denoting $p^\nu_o=a^2(\lambda_o)/v^\nu_o$,

\begin{enumerate}
\item $\ds \tv_{}\left(\log(p^\nu_o)\right)\le \tv_{} \left(\log(p_o)\right)$, $\ds \tv_{}\left(u^\nu_o\right) \le \tv_{}\left(u_o \right)$;

\item $\lim_{x\to-\infty} (v^\nu_o,u^\nu_o)(x)
  =\lim_{x\to-\infty} (v_o,u_o)(x)$;

\item $\|(v^\nu_o,u^\nu_o) - (v_o,u_o)\|_{\L1}\leq 1 /\nu$.
\end{enumerate}

We introduce two strictly positive parameters: $\sigma=\sigma_\nu$, that controls the size of rarefactions, and a threshold $\rho=\rho_\nu$,
that determines which of the two Riemann solver is to be used and depends on the initial data. Here follows a description of the scheme that improves the algorithm of  \cite{amadori-corli-siam} and adapts it to the current situation.

\begin{enumerate}[{\em (i)}]

\item At time $t=0$ we solve the Riemann problems 
at each point of jump of $(v^\nu_o, u^\nu_o, \lambda_o)(\cdot, 0+)$ as follows: shocks are not modified while rarefactions are approximated by fans of waves, each of them having size
less than $\sigma$. More precisely, a rarefaction of size $\eps$ is approximated
by $N=[\eps/\sigma]+1$ waves whose size is $\eps/N<\sigma$; we set their speeds to
be equal to the characteristic speed of the state at the right.
Then $(v^\nu,u^\nu,\lambda_o)(\cdot,t)$ is defined until some wave fronts interact; by slightly changing the speed of some waves we can assume that only \emph{two} fronts interact at a time.

\item When two wave fronts of the families $1$ or $3$ interact, we solve the Riemann problem at the interaction point. If one of the incoming waves is a rarefaction, after the interaction it is prolonged (if it still exists) as a single discontinuity with speed equal to the characteristic speed of the state at the right.  If a new rarefaction is generated, we employ the Riemann solver described in step {\em (i)} and split the rarefaction into a fan of waves having size less than $\sigma$.

\item When a wave front of family $1$ or $3$ with strength $\delta$ interacts with one of the composite waves at a time $t>0$, we proceed as follows:
    \begin{itemize}
    \item if $|\delta|\ge\rho$, we use the {\em Accurate solver} introduced in
    Proposition~\ref{prop:accsimpl}, partitioning the possibly new rarefaction
    according to \emph{(i)};
    \item if $|\delta|<\rho$, we use the {\em Simplified solver} of Proposition~\ref{prop:accsimpl}.
    \end{itemize}

\end{enumerate}


\section{Interactions}\label{sec:interaction}
\setcounter{equation}{0}

In this section we analyze interactions between waves. We separately study interactions that involve one of the two composite waves and interactions between $3$- and $1$-waves occurring in one of the regions $\mathcal{L},\mathcal{M},\mathcal{R}$. In particular, we focus on the interaction estimates for the former ones and we introduce a new functional $F$, different from that of \cite{amadori-corli-siam}, to estimate the possible increase of the total variation.

Consider $t>0$ at which no interactions occur and $\xi\ge1$ to be determined. Using indices $\ell,m,r$ to refer to waves in the region $\mathcal{L}, \mathcal{M}, \mathcal{R}$, respectively, we define $L= L^{\ell}+L^m+L^r$, where
\begin{equation*}
L^{\ell,m,r}=\sum_{\genfrac{}{}{0pt}{}{i=1,3,\,\delta_i>0}{\delta_i\in\mathcal{L},\mathcal{M},\mathcal{R}}}|\delta_i| + \xi \sum_{\genfrac{}{}{0pt}{}{i=1,3,\,\delta_i<0}{{\delta_i\in\mathcal{L},\mathcal{M},\mathcal{R}}}}|\delta_i|\,.
\end{equation*}
For $K_{\eta,\zeta}^{\ell,m,r}>0$ (see Figure~\ref{fig:droplambda}) we introduce $Q=Q^\ell+Q^m+Q^r$, where
\begin{align*}
Q^\ell& =\left(K_\eta^\ell|\eta|+ K_\zeta^\ell|\zeta|\right)\sum_{\genfrac{}{}{0pt}{}{\delta_3>0}{\delta_3\in\mathcal{L}}}|\delta_3|+\xi K_\eta^\ell\sum_{\genfrac{}{}{0pt}{}{\delta_3<0}{\delta_3\in\mathcal{L}}}|\delta_3\eta|\,,	 \\
Q^{m}&=K_\eta^m\sum_{\genfrac{}{}{0pt}{}{\delta_1>0}{\delta_1\in\mathcal{M}}}|\delta_1\eta| + K_\zeta^m \sum_{\genfrac{}{}{0pt}{}{\delta_3>0}{\delta_3\in\mathcal{M}}}|\delta_3 \zeta|\,,\\
Q^r &= \left(K_\eta^r|\eta|+ K_\zeta^r|\zeta|\right)\sum_{\genfrac{}{}{0pt}{}{\delta_1>0}{\delta_1\in\mathcal{R}}}|\delta_1| +\xi K_\zeta^r\sum_{\genfrac{}{}{0pt}{}{\delta_1<0}{\delta_1\in\mathcal{R}}}|\delta_1\zeta|\,.
\end{align*}
Moreover, we define $F^{\ell,m,r}=L^{\ell,m,r}+Q^{\ell,m,r}$ and
\begin{equation}\label{F}
F =  F^\ell + F^m +F^r + L^0\,,
\end{equation}
where $L^0=|\eta_0|+|\zeta_0|$. Clearly, $F$ can be seen as defined also by $F=L+Q+L^0$. We also write
\begin{equation*}
\bar{L}  =  \bar{L}^\ell+\bar{L}^m+\bar{L}^r= \sum_{\genfrac{}{}{0pt}{}{i=1,3}{\delta_i\in\mathcal{L}}}|\delta_i|+\sum_{\genfrac{}{}{0pt}{}{i=1,3}{\delta_i\in\mathcal{M}}}|\delta_i|+\sum_{\genfrac{}{}{0pt}{}{i=1,3}{\delta_i\in\mathcal{R}}}|\delta_i|= \frac 12 \tv\left(\log p(t,\cdot)\right)-|\eta_0|-|\zeta_0|\,.
\end{equation*}

\bigskip


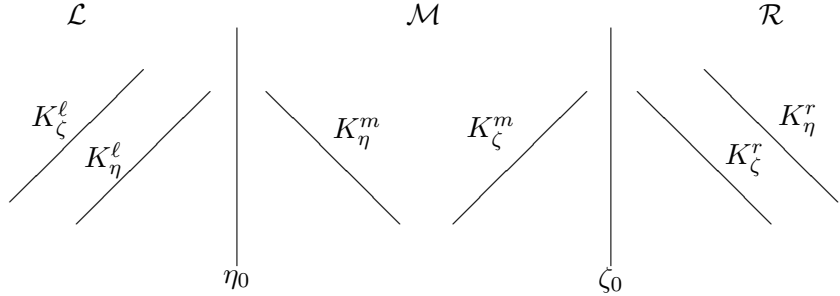
\begin{figure}[htbp]
\begin{picture}(100,110)(-120,10)
\setlength{\unitlength}{1pt}

\put(25,0){
\put(10,30){\line(0,1){90}}
\put(150,30){\line(0,1){90}}
\put(-75, 54){\line(1,1){50}}
\put(-50, 46){\line(1,1){50}}
\put(71, 46){\line(-1,1){50}}
\put(91, 46){\line(1,1){50}}
\put(210, 46){\line(-1,1){50}}
\put(235, 54){\line(-1,1){50}}


\put(-50,125){\makebox(0,0){$\mathcal{L}$}}
\put(80,125){\makebox(0,0){$\mathcal{M}$}}
\put(210,125){\makebox(0,0){$\mathcal{R}$}}

\put(-40,70){\makebox(0,0){$K_\eta^\ell$}}
\put(-60,85){\makebox(0,0){$K_\zeta^\ell$}}
\put(55,80){\makebox(0,0){$K_\eta^m$}}
\put(105,80){\makebox(0,0){$K_\zeta^m$}}
\put(200,70){\makebox(0,0){$K_\zeta^r$}}
\put(220,85){\makebox(0,0){$K_\eta^r$}}
\put(10,25){\makebox(0,0){$\eta_0$}}
\put(150,25){\makebox(0,0){$\zeta_0$}}
}
\end{picture}
\vspace{-4ex}
\caption{\label{fig:droplambda}{The parameters $K_{\eta,\zeta}^{\ell,m,r}$ related to the approached $2$-wave and to the regions of provenience of the approaching waves.}}
\end{figure}


\begin{remark}\label{rem:F_asymm}
The summation in $Q^\ell$ ($Q^r$) is performed over the set of waves {\em approaching} the composite waves from the left (right, respectively) but it does not include $3$-shocks approaching $\zeta_0$ ($1$-shocks approaching $\eta_0$, respectively). Moreover, the sum in $Q^m$ includes neither $3$-shocks approaching $\zeta_0$ nor $1$-shocks approaching $\eta_0$. Indeed, the contributions given by these waves can be dropped from the interaction potential since the linear functional $L$ decreases when they interact with a $0$-wave. The functional $F$ obtained in this way has proven to provide the best possible conditions on the parameters involved and, consequently, the largest ones on the initial data. Clearly, the choice of $F$ is reflected in \eqref{hyp2}, where the total variation of the data can be taken larger in $\mathcal{M}$ than in the outer regions.
\end{remark}

\begin{remark}\label{rem:why}
We will prove that the functional $F$ decreases when we use the Riemann solvers introduced in Proposition~\ref{prop:accsimpl}. This property does not hold true with the solvers of \cite{ABCD,amadori-corli-siam}. As already mentioned in Remark~\ref{rem:simpl}, the key point is that in the Simplified solver the strength of the transmitted wave is not the same of the incoming one, while they coincide for the Pseudo Simplified solver of \cite{ABCD}. Indeed, consider an asymmetrical functional $F_1$ adapted to the situation of \cite{ABCD}, i.e.\ $F_1$ does not include $3$-shocks in the interaction potential. In the case of an interaction of a $3$-shock with the $2$-wave $\delta_2>0$ solved by the Pseudo Simplified solver, we would get $\Delta L=K_{np}|\gamma_{2,0}|>0$ and $\Delta Q=0$. Thus, $F_1$ would increase.
\end{remark}

In the following we often assume that, for some fixed $m_o>0$, any interacting $i$-wave, $i=1,3$, with strength $\delta_i$ satisfies
\begin{equation}\label{rogna}
|\delta_i|\le m_o\,.
\end{equation}
In particular, this bound is to be imposed only to shock waves, since we can control the strength of the rarefaction waves by \eqref{eq:boundrar} below.

\subsection{Interactions with the composite waves}

Here we collect all the estimates concerning the composite waves.

\begin{lemma}[Interaction estimates]\label{lem:intest}
Let $i=1,3$. Consider the interaction of an $i$-wave $\delta_i$ with a composite wave $\delta_0=(\delta_0^1,\delta,\delta_0^3)$.
Denote by $\eps_i$ the strength of the transmitted wave and by $\eps_j$, $j=1,3,\,j\neq i$, the strength of the reflected one
(even in the Simplified case, where it is attached to $\delta_0$). Then, when $|\delta_i|\ge \rho$ it holds
\begin{equation}\label{eq:intest1}
|\eps_i-\delta_i|=|\eps_j|\le
\frac{1}{2}|\delta_i\delta| \qquad \text{and} \qquad |\eps_0-\delta_0|=0\,;
\end{equation}
while, when $|\delta_i|<\rho$ it holds
\begin{equation}\label{eq:intest2}
|\eps_i-\delta_i|=|\eps_0-\delta_0|=|\eps_j|\le\begin{cases}
\ds \frac{C_o}{2}|\delta_i\delta| & \text{if $\delta_i<0$ and either ($i=1$, $\delta>0$) or ($i=3$, $\delta<0$)}\,,\\[7pt]
\ds \frac{1}{2}|\delta_i\delta| & \text{otherwise}\,,
\end{cases}
\end{equation}
where $C_o=C_o(\rho)=\sinh(\rho)/\rho>1$ is such that $C_o(\rho)\to 1^+$ for $\rho\to 0^+$.
\end{lemma}

\begin{proof}
When $|\delta_i|\ge \rho$, i.e. when the Accurate solver is used, \eqref{eq:intest1}$_2$ is immediate and \eqref{eq:intest1}$_1$ can be derived from \eqref{eq:rels} (case $\theta_1=\theta_3=h$) following the same steps as in \cite[Theorem 2]{AmCo06-Proceed-Lyon}.

When $|\delta_i|<\rho$, i.e.\ when the Simplified solver is used, we analyze only the case $\delta>0$ and refer to Figure~\ref{fig:interactions} $(c),(d)$. The equality $|\eps_i-\delta_i|=|\eps_j|$ in \eqref{eq:intest2} is a consequence of \eqref{eq:rels}$_1$, while $|\eps_0-\delta_0|=|\eps_j|$ reflects our choice to attach the reflected wave to the composite one. To prove the inequality in \eqref{eq:intest2}, we distinguish cases according to the characteristic family and the sign of the interacting wave. If $\delta_i>0$, then $\RLI=\RII$ for $i=1$ and $\RIL=\RII$ for $i=3$; moreover, it holds $|\eps_j|=|\delta_i\delta|/2$ by \eqref{eq:rels}. When $i=3$ and the interacting wave has negative sign, as in Proposition~\ref{prop:accsimpl} we have $\RIL=\RLL$ and the interaction estimate \eqref{eq:intest2} follows exactly as in the Accurate case. Instead, when $i=1$ and the interacting wave has negative sign, we have to pay more attention. Recall from Proposition~\ref{prop:accsimpl} that we have $\eps_{1}=\eps_{3}+\delta_{1}$ and
$\eps_{3}<0$, $\delta_{1}<0$; moreover, \eqref{eq:delta1eq2} holds. By the Mean Value Theorem there exists some $s$ such that $\Gamma(\eps_3+\delta_1)-\Gamma(\delta_1)=\Gamma'(s)\eps_3$. Hence, we have
$$
(k+1)|\eps_3|\le \Gamma'(s)|\eps_3|=(k-1)\sinh|\delta_1|
$$
and we deduce
$$
|\eps_3|\le \frac{k-1}{k+1}\sinh|\delta_1|=\frac{\delta}{2}\sinh|\delta_1|\le \frac{C_o}{2}|\delta_1\delta|\,.
$$
\end{proof}


We will discuss in Appendix~\ref{appB} a refinement of estimate \eqref{eq:intest2}. Notice also  that in the previous lemma the biggest effort is required to handle the estimates for shocks interacting with $\delta_0$ and going towards the phase where $a$ is smaller. In our case, these are precisely the shocks that hit $\mathcal{M}$ from the outside, i.e.\ $1$-shocks interacting with $\zeta_0$ and $3$-shocks with $\eta_0$.

Now, we are ready to give a first list of conditions to impose on the parameters $\xi,K_{\eta,\zeta}^{\ell,m,r}$ and $\rho$ in order that the functional $F$ decreases at any interaction time.

\begin{proposition}
    Assume that at a time $t>0$ a wave $\delta_i$, $i=1,3$, interacts with one of
    the composite waves $\eta_0$ or $\zeta_0$. Then, $\Delta F(t)\le 0$ provided
    that
    \begin{gather}\label{eq:composite}
	\xi\ge 1\,, \qquad
	K_\zeta^r,K_\eta^\ell\geq1\,, \qquad
	\frac{\xi-1}{2}\le K_\zeta^m\le\frac{\xi-1}{|\zeta|}\,, \qquad
	\frac{\xi-1}{2}\le K_\eta^m\le\frac{\xi-1}{|\eta|}\,,\\
	\ds K_\eta^m\bigl(1+\frac{|\zeta|}{2}\bigr)|\eta|\le
	K_\eta^r|\eta|+(K_\zeta^r-1)|\zeta|\,,\qquad
	\ds K_\zeta^m\bigl(1+\frac{|\eta|}{2}\bigr)|\zeta|\le
	K_\zeta^\ell|\zeta|+(K_\eta^\ell-1)|\eta|\,,\label{eq:composite2}\\
	 C_{o}(\rho)\le\frac{2\xi}{\xi+1}\min\{K_\zeta^r,K_\eta^\ell\}\,.\label{eq:composite3}
    \end{gather}
\end{proposition}
\begin{proof}


\begin{figure}[htbp]
\begin{picture}(100,100)(-80,-15)
\setlength{\unitlength}{1pt}

\put(50,0){
\put(0,0){\line(0,1){40}} \put(-2,-5){\makebox(0,0){$\zeta_0$}}
\put(-60,0){\line(0,1){40}} \put(-62,-5){\makebox(0,0){$\eta_0$}}
\put(20,0){\line(-1,2){20}} \put(20,-5){\makebox(0,0){$\delta_1$}}
\put(0,40){\line(0,1){30}} \put(-2,75){\makebox(0,0){$\zeta_0$}}
\put(-60,40){\line(0,1){30}} 
\put(0,40){\line(1,2){15}} \put(20,75){\makebox(0,0){$\eps_3$}}
\put(0,40){\line(-1,1){30}} \put(-30,75){\makebox(0,0){$\eps_1$}}
\put(-37,40){\makebox(0,0){$\mathcal{M}$}}
\put(40,40){\makebox(0,0){$\mathcal{R}$}}
\put(-90,40){\makebox(0,0){$\mathcal{L}$}}
\put(-37,-22){\makebox(0,0){$(a)$}}
}

\put(270,0){
\put(0,0){\line(0,1){40}} \put(-2,-5){\makebox(0,0){$\zeta_0$}}
\put(-20,0){\line(1,2){20}}\put(-20,-5){\makebox(0,0){$\delta_3$}}
\put(-60,0){\line(0,1){40}} \put(-62,-5){\makebox(0,0){$\eta_0$}}
\put(-60,40){\line(0,1){30}} 
\put(0,40){\line(0,1){30}} \put(-2,75){\makebox(0,0){$\zeta_{0}$}}
\put(0,40){\line(1,1){30}} \put(30,75){\makebox(0,0){$\eps_3$}}
\put(0,40){\line(-1,2){15}} \put(-20,75){\makebox(0,0){$\eps_1$}}
\put(-37,40){\makebox(0,0){$\mathcal{M}$}}
\put(40,40){\makebox(0,0){$\mathcal{R}$}}
\put(-90,40){\makebox(0,0){$\mathcal{L}$}}
\put(-37,-22){\makebox(0,0){$(b)$}}
}

\end{picture}
\vspace{10pt}
\caption{\label{fig:inter2}{Interactions of $1$- and $3$-waves with $\zeta_0$ solved by means of the Accurate solver. Here the fronts carrying the composite waves are represented as a single line.}}
\end{figure}
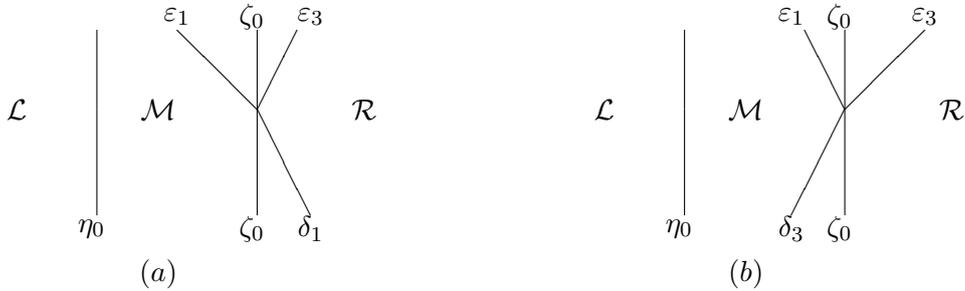


Since the two cases give symmetric conditions, we only analyze interactions involving $\zeta_0$; see Figure~\ref{fig:inter2}. 
We have
\begin{equation*}
\left\{
\begin{array}{lll}
\eps_3-\eps_1 = -\delta_1, &\quad |\eps_1|-|\delta_1| = |\eps_3|\,,  &\qquad \mbox{if  }i=1\,, \\[1mm]
\eps_3-\eps_1 = \delta_3, &\quad |\eps_3|-|\delta_3|  =  -|\eps_1|\,, &\qquad \mbox{if  }i=3\,.
\end{array}
\right.
\end{equation*}

\smallskip\noindent{\fbox{$i=1$.}} If the interacting wave is a rarefaction,
then by \eqref{eq:intest1},\eqref{eq:intest2}
we have
$\Delta L+\Delta L^0=|\eps_3|+|\eps_1|-|\delta_1|=2|\eps_3|\le |\delta_1\zeta|$
and $\Delta Q = K_\eta^m|\eps_1\eta|- K_\eta^r|\delta_1\eta|
-K_\zeta^r|\delta_1\zeta|$.
Therefore,
$$
\Delta F
\le |\delta_{1}|\left[\Big(K_\eta^m\bigl(1+\frac{|\zeta|}{2}\bigr)
-K_\eta^r\Big)|\eta|+(1-K_\zeta^r)|\zeta|\right]\,,
$$
which is nonpositive by \eqref{eq:composite2}$_{1}$.
Instead, if the interacting wave is a shock, then by \eqref{eq:intest1},\eqref{eq:intest2}
$$
\Delta L+\Delta L^0= \begin{cases}
\ds\xi|\eps_1|+\xi|\eps_3|-\xi|\delta_1|=2\xi|\eps_3| \le \xi|\delta_1\zeta| &\quad \text{if $|\delta_1|\ge \rho$}\,,\\[7pt]
\ds\xi|\eps_1|+|\eps_3|-\xi|\delta_1|=(1 +\xi)|\eps_3|
\le C_{o}(1 +\xi)\frac{|\delta_1\zeta|}{2}&\quad \text{if $|\delta_1|< \rho$}\,,
\end{cases}
$$
and $\Delta Q =  - K_\zeta^r\xi|\delta_1\zeta|$ in both cases. Consequently,
\begin{equation*}
\Delta F
\le \begin{cases}
\ds \xi \left[1-K_\zeta^r\right]|\delta_1\zeta|  &\quad \text{if $|\delta_1|\ge \rho$}\,,\\[7pt]
\ds \left[(1+\xi)\frac{C_{o}}{2}-\xi K_\zeta^r\right]|\delta_1\zeta| &\quad \text{if $|\delta_1|<\rho$}\,,
\end{cases}
\end{equation*}
is nonpositive by \eqref{eq:composite}$_{1,2}$ and \eqref{eq:composite3}.

\smallskip\noindent{\fbox{$i=3$.}} If the interacting wave is a rarefaction,
then by the interaction estimates 
$$
\Delta L+\Delta L^0= \begin{cases}
\ds\xi|\eps_1| + |\eps_3| - |\delta_3| = (\xi -1 )|\eps_1| \le (\xi-1)\frac{|\delta_3\zeta|}{2} &\quad \text{if $|\delta_3|\ge \rho$}\,,\\[7pt]
\ds |\eps_1|+|\eps_3|-|\delta_3|=0 &\quad \text{if $|\delta_3|< \rho$}\,,
\end{cases}
$$
and $\Delta Q=-K_\zeta^m|\delta_3\zeta|$. Then,
\begin{equation*}
\Delta F \begin{cases}
\ds\le \left[\frac{\xi -1}{2}- K_\zeta^m\right]|\delta_3\zeta| &\quad \text{if $|\delta_3|\ge \rho$}\,,\\[7pt]
\ds = -K_\zeta^m|\delta_3\zeta| &\quad \text{if $|\delta_3|<\rho$}\,,
\end{cases}
\end{equation*}
is nonpositive if \eqref{eq:composite}$_{3}$ holds.
On the other hand, if the interacting wave is a shock,
then $\Delta L+\Delta L^0= |\eps_1|+\xi|\eps_3|-\xi|\delta_3|=-(\xi-1)|\eps_1|$
and
$$
\Delta Q=\begin{cases}
\ds K_\eta^m|\eps_1\eta| &\quad \text{if $|\delta_3|\ge \rho$}\,,\\[7pt]
0 &\quad \text{if $|\delta_3|< \rho$}\,.
\end{cases}
$$
Therefore,
\begin{equation*}
\Delta F = \begin{cases}
\ds \left[-(\xi-1)+K_\eta^m|\eta|\right]|\eps_1| &\quad \text{if $|\delta_3|\ge \rho$}\,,\\[7pt]
\ds -(\xi-1)|\eps_1| &\quad \text{if $|\delta_3|<\rho$}\,,
\end{cases}
\end{equation*}
is nonpositive by \eqref{eq:composite}$_{4}$.
\end{proof}


\subsection{Interactions between waves of the same family}

In this subsection we analyze the interactions between $1$- and $3$-waves, see Figure~\ref{fig:inter3133}.


\begin{figure}[htbp]
\begin{picture}(100,80)(-130,-15)
\setlength{\unitlength}{0.8pt}

\put(-20,0){
\put(0,40){\line(-2,-3){30}}\put(-40,0){\makebox(0,0){$\delta_3$}}
\put(20,0){\line(-1,2){20}} \put(30,0){\makebox(0,0){$\delta_1$}}
\put(0,40){\line(1,2){15}} \put(20,75){\makebox(0,0){$\eps_3$}}
\put(0,40){\line(-1,1){30}} \put(-30,75){\makebox(0,0){$\eps_1$}}
\put(0,-20){\makebox(0,0){$(i)$}}
}

\put(140,0){
\put(0,40){\line(1,-1){45}}
\put(15,0){\makebox(0,0){$\alpha_1$}}
\put(0,40){\line(2,-3){30}}
\put(55,0){\makebox(0,0){$\beta_1$}}
\put(0,40){\line(1,2){15}} \put(20,75){\makebox(0,0){$\eps_3$}}
\put(0,40){\line(-1,1){30}} \put(-30,75){\makebox(0,0){$\eps_1$}}
}
%
\put(300,0){ \put(0,40){\line(-1,-1){45}}
\put(-55,0){\makebox(0,0){$\alpha_3$}}
\put(0,40){\line(-2,-3){30}}\put(-10,0){\makebox(0,0){$\beta_3$}}
\put(0,40){\line(1,2){15}} \put(20,75){\makebox(0,0){$\eps_3$}}
\put(0,40){\line(-1,1){30}} \put(-30,75){\makebox(0,0){$\eps_1$}}
\put(-80,-20){\makebox(0,0){$(ii)$}}
}

\end{picture}

\caption{\label{fig:inter3133}{Interactions of $1$- and $3$-waves.}}
\end{figure}
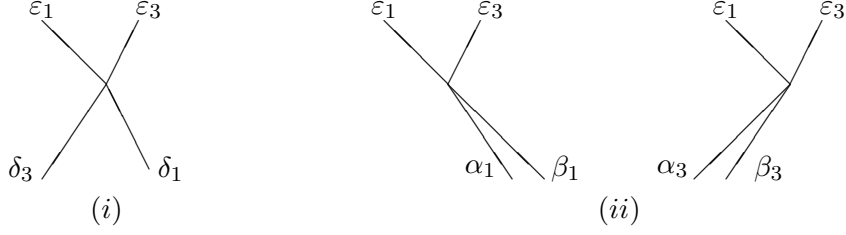


For interactions of two $i$-waves, $i=1,3$, under the notation of Figure~\ref{fig:inter3133} we shall make use of the identities \cite{AmCo06-Proceed-Lyon}:
\begin{equation}\label{eq:impo3}
\eps_3-\eps_1=\begin{cases}
-\alpha_1 -\beta_1 &\quad \text{if $i=1$}\,,\\
\alpha_3 +\beta_3 &\quad \text{if $i=3$}\,,
\end{cases}
\qquad h(\eps_1)+h(\eps_3)=h(\alpha_i)+h(\beta_i)\,,\quad i=1,3\,.
\end{equation}

\begin{lemma}
\label{lem:shock-riflesso}
For the interaction patterns in Figure~\ref{fig:inter3133}, the following holds.

\begin{enumerate}[{(i)}]
\item Two interacting waves of different families
cross each other without changing strengths.

\item Let $\alpha_i$, $\beta_i$ be two interacting waves of the same family and $\eps_1$, $\eps_3$ the outgoing waves.

\begin{enumerate}[({ii}.a)]

\item  If both incoming waves are shocks, then the outgoing wave of the same family is a shock and satisfies $|\eps_i|>\max \{|\alpha_i|,|\beta_i|\}$; the reflected wave is a rarefaction.

\item  If the incoming waves have different signs, then the reflected wave is a shock; both the amounts of shocks and rarefactions of the $i$-th family decrease across the interaction.  Moreover for $j\ne i $ and $\alpha_i<0<\beta_i$ one has
    \begin{align}\label{eq:chi_def}
    |\eps_j| & \le c(\alpha_i) \cdot \min\{|\alpha_i|,|\beta_i|\}\,,\qquad c(z) = \frac{\cosh z -1}{\cosh z+1}\,.
    \end{align}
\end{enumerate}
\end{enumerate}
\end{lemma}

The proof of Lemma~\ref{lem:shock-riflesso} can be found in \cite{ABCD}, where the function $c$ is used in place of the \emph{damping coefficient} $d$ of \cite{amadori-corli-siam}. Remark also that, by definition of the functionals, we need to distinguish between interactions taking place in $\mathcal{M}$ and in $\mathcal{L}$ or $\mathcal{R}$.

\begin{proposition}
    Consider the interaction at time $t>0$ of two waves of the same family $1$
    or $3$ and assume \eqref{rogna}. Then, $\Delta F(t)\le 0$ provided that
    \begin{gather}\label{cond13}
	1\le \xi \le \frac{1}{c(m_o)}\,, \qquad
	K_\zeta^m\le \frac{\xi-1}{|\zeta|}\,,\qquad
	K_\eta^m\le \frac{\xi-1}{|\eta|}\,,\\
	K_\eta^r|\eta|+K_\zeta^r|\zeta|\le \xi-1\,,\qquad
	K_\eta^\ell|\eta|+K_\zeta^\ell|\zeta|\le \xi-1\,.\label{cond13bis}
    \end{gather}
\end{proposition}

\begin{proof}
First, we consider the interactions taking place in $\mathcal{M}$, see Figure~\ref{fig:inter4}. Here, we only cover the case of interactions between two $3$-waves $\alpha_3$ and $\beta_3$ giving rise to $\eps_1$ and $\eps_3$ (the $1$-waves case is analogous).


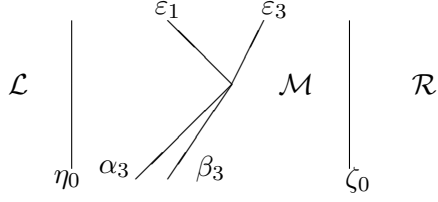
\begin{figure}[htbp]
\begin{picture}(100,80)(-130,-15)
\setlength{\unitlength}{0.8pt}

\put(130,0){
\put(-75,0){\line(0,1){40}} \put(-77,-5){\makebox(0,0){$\eta_0$}}
\put(-75,40){\line(0,1){30}} 
\put(55,0){\line(0,1){40}} \put(59,-5){\makebox(0,0){$\zeta_0$}}
\put(55,40){\line(0,1){30}} 
\put(0,40){\line(-1,-1){45}}\put(-55,0){\makebox(0,0){$\alpha_3$}}
\put(0,40){\line(-2,-3){30}}\put(-10,0){\makebox(0,0){$\beta_3$}}
\put(0,40){\line(1,2){15}} \put(20,75){\makebox(0,0){$\eps_3$}}
\put(0,40){\line(-1,1){30}} \put(-30,75){\makebox(0,0){$\eps_1$}}
\put(-100,40){\makebox(0,0){$\mathcal{L}$}}
\put(30,40){\makebox(0,0){$\mathcal{M}$}}
\put(90,40){\makebox(0,0){$\mathcal{R}$}}
}

\end{picture}

\caption{\label{fig:inter4}{Interactions of $3$-waves in $\mathcal{M}$.}}
\end{figure}


\noindent When both $\alpha_3$ and $\beta_3$ are shocks, by Lemma~\ref{lem:shock-riflesso} we have that $\eps_1$ is a rarefaction and we notice as in \cite[Proposition 5.8]{ABCD} that
\begin{equation}\label{Delta_L_xi_13}
\Delta L+|\eps_1|(\xi-1)=0\,,
\end{equation}
for any $\xi\ge 1$. Moreover, we have
\begin{align*}
\Delta Q&=K_\eta^m|\eps_{1}\eta|\,,\\
\Delta F&=\left[-(\xi-1)+K_\eta^m|\eta|\right]|\eps_{1}|
\end{align*}
and $F$ is non-increasing by \eqref{cond13}$_{1,3}$.
On the other hand, when the two interacting waves are of different type,
for example $\alpha_3<0<\beta_3$, as in \cite[Proposition 5.8]{ABCD} one
can prove that it holds
\begin{equation}\label{Delta_L_xi_13_SR}
\Delta L+\xi(\xi-1)|\eps_1|\le 0
\end{equation}
by condition \eqref{cond13}$_{1}$. If $\eps_3$ is a rarefaction, then
$\Delta Q= K_\zeta^m\left(|\eps_{3}|-|\beta_{3}|\right)|\zeta|$ and $F$
decreases by Lemma~\ref{lem:shock-riflesso}; if $\eps_3$ is a shock,
then $\Delta Q= -K_\zeta^m|\beta_{3}\zeta|$ and, again, $F$ decreases.
Remark that the analysis of the interactions between $1$-waves requires
symmetrically the condition $K_\zeta^m\le(\xi-1)/|\zeta|$.

Next, we analyze the case of interactions taking place in $\mathcal{R}$
(similarly one proceeds to analyze those occurring in $\mathcal{L}$).
As for interactions between two $1$-waves, it is easy to verify that $F$
decreases with no need of other conditions than \eqref{cond13}$_{1}$.


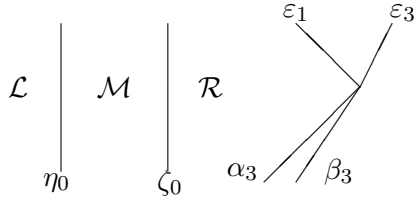
\begin{figure}[htbp]
\begin{picture}(100,80)(-130,-15)
\setlength{\unitlength}{0.8pt}

\put(190,0){
\put(-140,0){\line(0,1){40}} \put(-142,-5){\makebox(0,0){$\eta_0$}}
\put(-140,40){\line(0,1){30}} 
\put(-90,0){\line(0,1){40}} \put(-89,-5){\makebox(0,0){$\zeta_0$}}
\put(-90,40){\line(0,1){30}} 
\put(0,40){\line(-1,-1){45}}\put(-55,0){\makebox(0,0){$\alpha_3$}}
\put(0,40){\line(-2,-3){30}}\put(-10,0){\makebox(0,0){$\beta_3$}}
\put(0,40){\line(1,2){15}} \put(20,75){\makebox(0,0){$\eps_3$}}
\put(0,40){\line(-1,1){30}} \put(-30,75){\makebox(0,0){$\eps_1$}}
\put(-160,40){\makebox(0,0){$\mathcal{L}$}}
\put(-115,40){\makebox(0,0){$\mathcal{M}$}}
\put(-70,40){\makebox(0,0){$\mathcal{R}$}}
}

\end{picture}

\caption{{Interactions of $3$-waves in $\mathcal{R}$.}\label{fig:inter5}}
\end{figure}


\noindent As for interactions between $3$-waves (Figure~\ref{fig:inter5}),
instead, we need to require \eqref{cond13}$_{1,3}$ in order to have
$\Delta F(t)\le 0$. Indeed, we have \eqref{Delta_L_xi_13} when the
interacting waves $\alpha_3,\beta_3$ are both shocks, while in the
other two cases of interaction it still holds \eqref{Delta_L_xi_13_SR}
under condition \eqref{cond13bis}$_1$. Also, if $\alpha_3,\beta_3<0$ we have
\begin{align*}
\Delta Q&=K_\eta^r|\eps_1\eta|+K_\zeta^r|\eps_{1}\zeta|\,,\\
\Delta F&= \left[-(\xi-1)+K_\eta^r|\eta|+K_\zeta^r|\zeta|\right]|\eps_1|\,,
\end{align*}
while if, for example, $\alpha_3<0<\beta_3$ we have
\begin{align*}
\Delta Q&=K_\zeta^r\,\xi|\eps_1\zeta|\,,\\
\Delta F&\le \xi\left[-(\xi-1)+K_\zeta^r|\zeta|\right]|\eps_1|\,.
\end{align*}
Consequently, $F$ is non-increasing by \eqref{cond13bis}$_{1}$.

Symmetrically, in $\mathcal{L}$ we get the condition
$K_\eta^\ell|\eta|+K_\zeta^\ell|\zeta|\le \xi-1$.
\end{proof}


\subsection{Decreasing of the functional $F$ and control of the variations}\label{sec:decreasingF}

In order that $F$ decreases across any interaction, the various parameters in \eqref{rogna}, \eqref{eq:composite}, \eqref{eq:composite2}, \eqref{eq:composite3}, \eqref{cond13} and \eqref{cond13bis} are chosen in the following order. Given $m_o$ to be fixed later on, we choose in turn $\xi$, $K_{\eta,\zeta}^m$, $K_{\eta,\zeta}^{\ell,r}$ and finally $\rho$. Remark that in the following calculations we keep (almost) everywhere strict inequalities, since they are needed in the analysis on the control of the size of the composite waves (see Section~\ref{generationorder}). 

We notice that, for the choice of $K_{\zeta,\eta}^m$, by \eqref{eq:composite}$_{3,4}$ and \eqref{cond13}$_{}$ it must hold
\begin{equation}\label{eq:intervalKm}
\frac{\xi-1}{2}<\min\left\{\frac{\xi-1}{|\eta|},\frac{\xi-1}{|\zeta|}\right\}\,,
\end{equation}
which is always satisfied since $|\eta|,|\zeta|<2$. Moreover, by putting together the conditions obtained in \eqref{eq:composite}$_{3}$ with \eqref{eq:composite2}$_1$ we get necessarily
\begin{equation}\label{Ks-cond}
(\xi-1)\bigl(1+\frac{|\zeta|}{2}\bigr)\frac{|\eta|}{2}
< K_\eta^r|\eta|+(K_\zeta^r-1)|\zeta|
<(\xi-1)-K_\zeta^r|\zeta|\le (\xi-1)-|\zeta|\,.
\end{equation}
Hence, it follows
\begin{equation*}
(\xi-1)\bigl(1+\frac{|\zeta|}{2}\bigr)\frac{|\eta|}{2}<(\xi-1)-|\zeta|\,,
\end{equation*}
which is equivalent to
\begin{equation}\label{eq:cxi2}
1+\frac{|\zeta|}{1-(1+{|\zeta|}/{2}){|\eta|}/{2}}<\xi\,,
\end{equation}
provided that $1-(1+{|\zeta|}/{2}){|\eta|}/{2}>0$. Analogously, from \eqref{eq:composite}$_4$ and \eqref{eq:composite2}$_1$ we get
$$
1+\frac{|\eta|}{1-(1+{|\eta|}/{2}){|\zeta|}/{2}}<\xi\,,
$$
provided that $1-(1+{|\eta|}/{2}){|\zeta|}/{2}>0$. Therefore, it must hold
\begin{equation}\label{eq:H2}
1+\max\left\{\frac{|\zeta|}{1-(1+{|\zeta|}/{2}){|\eta|}/{2}},
\frac{|\eta|}{1-(1+{|\eta|}/{2}){|\zeta|}/{2}}\right\}<\xi
\end{equation}
under the \emph{stability} condition
$$
\min\left\{1-\bigl(1+\frac{|\zeta|}{2}\bigr)\frac{|\eta|}{2},
1-\bigl(1+\frac{|\eta|}{2}\bigr)\frac{|\zeta|}{2}\right\}>0\,,
$$
which is equivalent to \eqref{stab}. Then, by \eqref{eq:Hdef}, \eqref{eq:composite}$_1$ and \eqref{eq:H2} we obtain the condition
\begin{equation*}
      1+\mathcal{H}(|\eta|,|\zeta|)< \xi \le \frac{1}{c(m_o)}\,.
\end{equation*}

Summarizing, the choice of the parameters proceeds as follows. Let $|\eta|,|\zeta|$ satisfy \eqref{stab}.
\begin{itemize}
\item Recalling \eqref{cond13}$_1$, we fix $m_o$ such that
\begin{equation}\label{eq:cmo}
c(m_o)<\frac{1}{1+\mathcal{H}(|\eta|,|\zeta|)}\,.
\end{equation}
We will prove in Section~\ref{subsec:proof} that this is possible under suitable conditions on the initial data. Notice that $c$ is a strictly increasing function of $m_o$ and then it is invertible.
\item Then, we choose $\xi$ in the non-empty interval given by
\begin{equation}\label{eq:cmxi}
1+\mathcal{H}(|\eta|,|\zeta|)<\xi \le\frac{1}{c(m_o)}\,.
\end{equation}
\item We choose $K_\eta^m,K_\zeta^m$ such that
\begin{align}
    \frac{\xi-1}{2} < K_\eta^m<\min\left\{\frac{\xi-1}{|\eta|},
    \frac{(\xi-1)-|\zeta|}{(1+|\zeta|/2)|\eta|}\right\}
    =\frac{(\xi-1)-|\zeta|}{(1+|\zeta|/2)|\eta|}\,,\label{eq:cndKmeta}\\[6pt]
    \frac{\xi-1}{2} < K_\zeta^m<\min\left\{\frac{\xi-1}{|\zeta|},
    \frac{(\xi-1)-|\eta|}{(1+|\eta|/2)|\zeta|}\right\}
    =\frac{(\xi-1)-|\eta|}{(1+|\eta|/2)|\zeta|}\,.\label{eq:cndKmzeta}
\end{align}
This is possible since these two intervals are non-empty by \eqref{eq:cmxi} and \eqref{eq:cxi2}. Thus, \eqref{eq:composite}$_{3,4}$ and \eqref{cond13}$_{2,3}$ follow and it holds
\begin{align}\label{eq:kmetazeta2}
     K_\eta^m\bigl(1+\frac{|\zeta|}{2}\bigr)|\eta|<(\xi-1)-|\zeta|\,, \qquad K_\zeta^m\bigl(1+\frac{|\eta|}{2}\bigr)|\zeta|<(\xi-1)-|\eta|\,.
\end{align}
\item By \eqref{eq:kmetazeta2}, we choose $K_\eta^r,K_\zeta^\ell$ that satisfy
\begin{align}
    K_\eta^m\bigl(1+\frac{|\zeta|}{2}\bigr)|\eta|\le K_\eta^r|\eta|
    <(\xi-1)-|\zeta|\,, \label{eq:cndKreta}\\[6pt]
    K_\zeta^m\bigl(1+\frac{|\eta|}{2}\bigr)|\zeta|\le K_\zeta^\ell|\zeta|
    <(\xi-1)-|\eta|\,;\label{eq:cndKlzeta}
\end{align}
then, we can take $K_\eta^\ell$ and $K_\zeta^r$ such that
\begin{align}
	 1&<K_\zeta^r<1+\frac{(\xi-1)-|\zeta|-K_\eta^r|\eta|}{|\zeta|}\,,\label{eq:cndKrzeta}\\
	 1&<K_\eta^\ell<1+\frac{(\xi-1)-|\eta|-K_\zeta^\ell|\zeta|}{|\eta|}\,,\label{eq:cndKleta}
\end{align}
from which \eqref{eq:composite2} and \eqref{cond13bis} follow.
\item Finally, we choose $\rho$ that satisfies \eqref{eq:composite3}.
\end{itemize}

In the following proposition we collect the results obtained so far.

\begin{proposition}[Local decreasing]\label{prop:last}
Let $m_o > 0$ satisfy \eqref{eq:cmo} and consider the interaction of two waves at time $t$ that satisfy \eqref{rogna}.
Moreover, assume that $\xi$, $K_{\eta,\zeta}^{\ell,m,r}$ and $\rho$ satisfy \eqref{eq:cmxi}--\eqref{eq:cndKleta} and \eqref{eq:composite3}. Then,
\begin{equation}\label{eq:Fdecr}
\Delta F(t) \le 0\,.
\end{equation}
\end{proposition}

Now, we can prove the global in time decreasing of the functional $F$.

\begin{proposition}[Global decreasing]\label{global}
Let $m_o>0$ satisfy \eqref{eq:cmo} and choose parameters $\xi$, $K_{\eta,\zeta}^{\ell,m,r}$ and $\rho$ as in Proposition~\ref{prop:last}.
Moreover, assume that
    \begin{equation}
	\label{eq:boundL-bis}
	\bar{L}^\ell(0) + c(m_o)\hspace{0.8pt}\bar{L}^m(0) +\bar{L}^r(0) \le m_o\hspace{0.8pt} c(m_o)\,,
    \end{equation}
and that the approximate solution is defined in $[0,T]$. Then, $F(0)\le m_o$, $\Delta F(t)\le 0$ for every $t\in(0,T]$ and \eqref{rogna} is satisfied.
\end{proposition}

\begin{proof}
For convenience, we use notation $L^m_{iR}$ and $L^m_{iS}$ to indicate the partial sums in $L^m$ due to $i$-rarefaction waves ($iR$) and $i$-shock waves ($iS$), respectively. By \eqref{eq:composite}$_{3,4}$ we have
\begin{align*}
F^m(0)&= L^m(0)+Q^m(0)\le\\
       &\le L^m_{1S}(0)+L^m_{1R}(0)\left(1+K_\eta^m|\eta|\right)+L^m_{3S}(0)+L^m_{3R}(0)\left(1+K_\zeta^m|\zeta|\right)\le \\
       &\le L^m_{1S}(0)+\xi L^m_{1R}(0)+L^m_{3S}(0)+ \xi L^m_{3R}(0)\le \xi\bar{L}^m(0)\,.
\end{align*}
Moreover, from \eqref{cond13bis} it follows
\begin{align*}
F^{\ell,r}(0)\le L^{\ell,r}(0)\left(1+K_{\eta}^{\ell,r}|\eta|+K_\zeta^{\ell,r}|\zeta|\right)\le \xi^2\bar{L}^{\ell,r}(0)\,.
\end{align*}
Then,
\begin{equation}\label{eq:gdproof}
F(0)=F^\ell(0)+F^m(0)+F^r(0)\le \xi^2\bar{L}^\ell(0)+ \xi\bar{L}^m(0)+ \xi^2\bar{L}^r(0)\,.
\end{equation}
Now, for a fixed $t\le T$, suppose by induction that $F(\tau)\le m_o$ and $\Delta F(\tau)\le 0$ for every $0<\tau<t$, interaction time. Then, by Proposition~\ref{prop:last} we have $\Delta F(t)\le 0$. This implies
$$
F(t)\le F(0)\le \xi^2\bar{L}^\ell(0)+ \xi\bar{L}^m(0)+ \xi^2\bar{L}^r(0)\,.
$$
By \eqref{eq:gdproof} and \eqref{eq:boundL-bis} the size $|\delta_i|$ of a shock ($i=1,3$) at time $t$ satisfies
$$
|\delta_i|\le  \frac{1}{\xi}F(t)\le \xi\bar{L}^\ell(0)+ \bar{L}^m(0)+ \xi\bar{L}^r(0)\le \frac{1}{c(m_o)}\bar{L}^\ell(0)+\bar{L}^m(0)+ \frac{1}{c(m_o)}\bar{L}^r(0)\le m_o\,.
$$
Hence, \eqref{rogna} is satisfied and the proof is complete.
\end{proof}


\section{The convergence and consistency of the algorithm}\label{sec:convergence}
\setcounter{equation}{0}

In this section we finally conclude the proof of Theorem~\ref{thm:main} on the convergence and consistency of the front tracking algorithm.

In order to be well-defined, the algorithm must satisfy three main requirements: i) the size of rarefaction waves must remain small; ii) the total number of wave fronts and interactions must be finite; iii) the total size of the composite waves must vanish as the approximation parameter $\nu$ tends to $+\infty$. 
The first one is accomplished as in \cite[Lemma $6.1$]{amadori-corli-siam} and, in particular,
the size $\eps$ of any rarefaction wave is bounded by
\begin{equation}\label{eq:boundrar}
0<\eps<\sigma\left(1+\frac{1}{2}\max\left\{|\eta|,|\zeta|\right\}\right)< 2\sigma\,.
\end{equation}
The second requirement can be proved as in \cite[Lemma $6.2$ and Proposition $6.3$]{amadori-corli-siam}; while the next section is devoted to the proof of iii), i.e.\ of the consistency of the algorithm.

\subsection{Control of the total size of the composite waves}\label{generationorder}

Our wave front-tracking scheme exploits the concept of generation order of a wave to prove that the size of the error attached to the two phase waves tends to zero. We introduce such generation order as in \cite{amadori-corli-siam} for $1$- and $3$-waves, while for the composite waves we proceed in the following way. Consider an interaction with $\delta_0$ (either $\eta_0$ or $\zeta_0$): we assign order $1$ to the component $\delta$ (which never changes) and order $k_{\gamma}+1$ to the outgoing $j$-th component 
when the interacting wave is an $i$-wave $\gamma$ of size $<\rho$, $i,j=1,3$, $j\neq i$, while we keep the order of the other component 
 unchanged. In other words, denote by $(k_{\delta_0^1},1,k_{\delta_0^3})$ the triple made of the orders of the three components and consider the interaction with a wave $\gamma$ solved by the Simplified solver. Then, for the outgoing $0$-wave $\eps_0=(\eps_0^1,\delta,\eps_0^3)$ we have
\begin{equation*}
(k_{\eps_0^1},1,k_{\eps_0^3})=\begin{cases}
(k_{\gamma}+1,1,k_{\delta_0^3}) &\quad \text{if $\gamma$ is of family $3$}\,,\\
(k_{\delta_0^1},1,k_{\gamma}+1) &\quad \text{if $\gamma$ is of family $1$}\,.
\end{cases}
\end{equation*}
Mimicking what has already been done in \cite{amadori-corli-siam,ABCD}, for every $k=1,2,\dots$ we define the functionals $L_k$, $Q_k$ and $F_k$ simply by referring $L,Q,F$ to waves with order $k$. Moreover, we define $\tilde{L}_k=\sum_{h\ge k}L_h$ and $\tilde{F}_k=\sum_{h\ge k}F_h$.
In detail, for $k\ge 1$ and for $\xi,K_{\eta,\zeta}^{\ell,m,r}$ as in Proposition~\ref{prop:last}, we define $L_k= L_k^{\ell}+L_k^m+L_k^r$, where
\begin{equation*}
L_k^{\ell,m,r}=\sum_{\genfrac{}{}{0pt}{}{i=1,3,\,\delta_i>0,\,k_{\delta_i}=k}{\delta_i\in\mathcal{L},\mathcal{M},\mathcal{R}}}|\delta_i| + \xi \sum_{\genfrac{}{}{0pt}{}{i=1,3,\,\delta_i<0,\,k_{\delta_i}=k}{{\delta_i\in\mathcal{L},\mathcal{M},\mathcal{R}}}}|\delta_i|
\end{equation*}
and $Q_k=Q_k^\ell+Q_k^m+Q_k^r$, where
{\allowdisplaybreaks \begin{align*}
Q_k^\ell &=\left(K_\eta^\ell|\eta|+K_\zeta^\ell|\zeta|\right)\sum_{\genfrac{}{}{0pt}{}{\delta_3>0,\,k_{\delta_3}=k}{\delta_3\in\mathcal{L}}}|\delta_3| +\xi K_\eta^\ell\sum_{\genfrac{}{}{0pt}{}{\delta_3<0.\,k_{\delta_3}=k}{\delta_3\in\mathcal{L}}}|\delta_3\eta|\,,	 \\
Q_k^{m}&=K_\eta^m\sum_{\genfrac{}{}{0pt}{}{\delta_1>0.\,\,k_{\delta_1}=k}{\delta_1\in\mathcal{M}}}|\delta_1\eta| + K_\zeta^m\sum_{\genfrac{}{}{0pt}{}{\delta_3>0,\,k_{\delta_3}=k}{\delta_3\in\mathcal{M}}}|\delta_3 \zeta|\,,\\
Q_k^r &= \left(K_\eta^r|\eta|+K_\zeta^r|\zeta|\right)\sum_{\genfrac{}{}{0pt}{}{\delta_1>0,\,k_{\delta_1}=k}{\delta_1\in\mathcal{R}}}|\delta_1| +\xi K_\zeta^r \sum_{\genfrac{}{}{0pt}{}{\delta_1<0,\,k_{\delta_1}=k}{\delta_1\in\mathcal{R}}}|\delta_1\zeta|\,.
\end{align*}}%
Moreover, we define \begin{equation}\label{Fk}
F_k  =  L_k+ Q_k+ L_k^0\,,
\end{equation}
where
\begin{equation}\label{Lk0}
L_k^{0}=\sum_{\tau_k < t}\left(|\Delta \eta_0|+|\Delta \zeta_0|\right)(\tau_k)\,,
\end{equation}
with $\tau_k$ in \eqref{Lk0} denoting the interaction times when a small reflected wave of order $k$ is born and attached to one of the composite waves. As a consequence, only the times $\tau_k$ where the Simplified solver is used give positive summands in \eqref{Lk0}, since when the Accurate solver is used we have $|\Delta \eta_0|+|\Delta \zeta_0|=0$.

By $I_k$ we denote the set of times in which two waves of the same family of order at most $k$ interact with each other, while by $J_k$ we denote the set of times in which a wave of order $k$ hits one of the two composite waves. Moreover, let ${\cal T}_k=I_k \cup J_k$. Here, we will prove the analogous of \cite[Proposition $6.1$]{ABCD}. In particular, we define:
\begin{equation}\label{eq:mu}
\begin{aligned}
  \mu  = \max\bigg\{&\frac{1}{2K_\eta^\ell-1}\,,\,\frac{1}{2K_\zeta^r-1}\,,\,\frac{\xi}{1+2K_\eta^m}\,,\,\frac{\xi}{1+2K_\zeta^m}\,,\,\frac{1+K_\eta^m|\eta|}{\xi}\,,\,\frac{1+K_\zeta^m|\zeta|}{\xi}\,,\,\\
  &\frac{1+(K_\eta^\ell|\eta|+K_\zeta^\ell|\zeta|)}{\xi}\,,\,\frac{1+(K_\eta^r|\eta|+K_\zeta^r|\zeta|)}{\xi}\,,\,\frac{C_o}{\xi(2K_\eta^\ell-1)}\,,\,\frac{C_o}{\xi(2K_\zeta^r-1)}\bigg\}.
\end{aligned}
\end{equation}
We have that $\mu<1$ by the conditions required in Proposition~\ref{prop:last}.

\begin{proposition}\label{prop:genorder}
Let $m_o,\xi,K_{\eta,\zeta}^{\ell,m,r}$ satisfy the assumptions of Proposition~\ref{prop:last} and assume that
\eqref{rogna} holds, that is, $|\delta_i|\le m_o$ for the size of every wave.
Then, the following holds for $\tau\in {\cal T}_h$, $h\ge 1$:
\begin{align}\label{Fk-segni-1}
&\Delta F_{h}<0\,,\qquad \Delta F_{h+1}>0\,,\\ \label{Fk-segni-2}
&\Delta F_k=0\qquad\,\, \, \hbox{ if }\ k\ge h+2\,.
\end{align}
Moreover,
\begin{equation}
\label{h=k-1b-XXX}
[\Delta F_{h+1}]_+ \le \mu \Bigl([\Delta F_h]_-   -
\sum_{\ell=1}^{h-1} \Delta F_\ell \Bigr)\,.
\end{equation}
\end{proposition}

\begin{remark}\label{ref:DeltaF<0} Proposition~\ref{prop:genorder} let us improve Proposition~\ref{prop:last}. Indeed, recalling that ${\cal T}_h= I_{h}\cup J_h$, Proposition~\ref{prop:genorder} implies, for $\tau \in I_{h}$,
\begin{equation*}
\Delta F = \sum_{\ell=1}^{h-1}\Delta F_{\ell} \,-\, [\Delta F_h]_- \,+\, [\Delta F_{h+1}]_+  \le -(1-\mu)[\Delta F_h]_- < 0 \,,
\end{equation*}
while for $\tau \in J_{h}$, being  $\sum_{\ell=1}^{h-1} [\Delta F_\ell]_+=0$, it gives
\begin{equation*}
\Delta F  \,=\, -  [\Delta F_h]_- \,+\, [\Delta F_{h+1}]_+ \le  -(1-\mu)[\Delta F_h]_- <0 \,.
\end{equation*}
Then, estimate \eqref{h=k-1b-XXX} specifies the decrease of the functional $F$ and improves \eqref{eq:Fdecr}.
\end{remark}

\begin{proof}[Proof of Proposition~\ref{prop:genorder}] If $k\ge h+2$, no wave of order $k$ is involved and then \eqref{Fk-segni-2} holds\,.
To prove \eqref{Fk-segni-1} and \eqref{h=k-1b-XXX}, we distinguish between two cases.


\smallskip\noindent{\fbox{$\tau\in I_{h}$}} (Interactions between waves of $1$-, $3$-family).

Clearly the $F_k$'s do not vary when a $1$-wave interacts with a $3$-wave. Then we consider interactions of waves of the same family occurring in one of the three distinct regions $\mathcal{L},\mathcal{M},\mathcal{R}$. Since $\tau\in I_{h}$, then $\Delta L_{h+1}=\Delta L_{h+1}^{\ell,m,r}>0$ and
\begin{equation}\label{eq:deltaQh+1}
0\le \Delta Q_{h+1}=\begin{cases}
\Delta Q^m_{h+1}\le \max\left\{K_\eta^m|\eta|,K_\zeta^m|\zeta|\right\}\Delta L^{m}_{h+1} & \text{for interactions in $\mathcal{M}$,}\\[20pt]
\Delta Q^{\ell,r}_{h+1}\le \left(K_\eta^{\ell,r}|\eta|+K_\zeta^{\ell,r}|\zeta|\right)\Delta L^{\ell,r}_{h+1} & \text{for interactions in $\mathcal{L},\mathcal{R}$.}
\end{cases}
\end{equation}
Also, $\Delta F_{h}=\Delta L_h+\Delta Q_h<0$, since both terms in the sum are negative or zero.
This proves \eqref{Fk-segni-1}.

By \eqref{Delta_L_xi_13} and \eqref{Delta_L_xi_13_SR} (see also \cite[(6.10)]{amadori-corli-siam}), we have that
\begin{equation}\label{0k-1}
[\Delta  L_{h+1}]_+ \, \le\,  \frac{1}{\xi} \Bigl([\Delta L_{h}]_- - \sum_{\ell = 1}^{h-1}\,\Delta L_\ell \Bigr)\,.
\end{equation}
From \eqref{0k-1}, \eqref{eq:deltaQh+1} and \eqref{eq:mu} we deduce that
\begin{equation}\label{eq:ll}
\begin{aligned}
0<\Delta F_{h+1}&\le
\begin{cases}
\left(1+\max\left\{K_\eta^m|\eta|,K_\zeta^m|\zeta|\right\}\right)[\Delta L_{h+1}]_+ & \text{for interactions in $\mathcal{M}$}\,,\\[7pt]
\left(1+K_\eta^{\ell,r}|\eta|+K_\zeta^{\ell,r}|\zeta|\right) [\Delta L_{h+1}]_+ & \text{for interactions in $\mathcal{L,R}$}\,,
\end{cases}\\
&\le \mu \Bigl( [\Delta L_{h}]_- - \sum_{\ell = 1}^{h-1}\,\Delta L_\ell\Bigr)\,.
\end{aligned}
\end{equation}
We now claim that
\begin{equation}
  [\Delta Q_{h}]_{-} - \sum_{\ell = 1}^{h-1}\Delta Q_\ell \ge 0\,.
\label{eq:mm}
\end{equation}
To prove \eqref{eq:mm}, we only have to analyze the cases when $\Delta Q_\ell>0$ for an $\ell\le h-1$. This can occur for interactions between waves $\alpha_i<0<\beta_i$ with $\alpha_i$ of order $\ell$ and $\beta_i$ of order $h$, giving rise to waves $\eps_i,\eps_j$ of different sign, $i,j=1,3$, $i\neq j$. More precisely, by Lemma~\ref{lem:shock-riflesso} we have
\begin{equation*}
[\Delta Q_{h}]_- - \sum_{\ell = 1}^{h-1}\Delta Q_\ell=
\begin{cases}
\left[\left(-|\eps_3|+|\beta_3|\right)\left(K_\eta^\ell|\eta|+K_\zeta^\ell|\zeta|\right)+\xi K_\eta^\ell|\alpha_3\eta|\right] &\quad \text{for interactions in $\mathcal{L}$}\,,\\[7pt]
K_\eta^m\left(-|\eps_1|+|\beta_1|\right)|\eta| &\quad \text{for interactions in $\mathcal{M}$}\,,\\[7pt]
K_\zeta^m\left(-|\eps_3|+|\beta_3|\right)|\zeta| &\quad \text{for interactions in $\mathcal{M}$}\,,\\[7pt]
\left[\left(-|\eps_1|+|\beta_1|\right)\left(K_\eta^r|\eta|+K_\zeta^r|\zeta|\right)+\xi K_\zeta^r|\alpha_1\zeta|\right] &\quad \text{for interactions in $\mathcal{R}$}\,,
\end{cases}
\end{equation*}
which is always a nonnegative quantity. This proves (\ref{eq:mm}). Therefore, for $\tau\in I_{h}$, estimate \eqref{h=k-1b-XXX} follows from (\ref{eq:ll}) and (\ref{eq:mm}).


\smallskip\noindent{\fbox{$\tau\in J_{h}$}}  (Interactions with the composite waves).

Here we focus only on interactions involving $\zeta_0$, since the other case gives symmetric conditions. Since no wave of order $\le h-1$ interact, then (\ref{h=k-1b-XXX}) reduces to
\begin{equation}\label{h=k-1b-XXX-caso-J}
[\Delta F_{h+1}]_+ \le \mu [\Delta F_{h}]_- \,.
\end{equation}


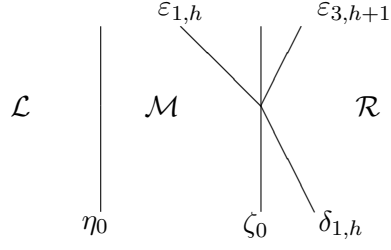
\begin{figure}[htbp]
\begin{picture}(100,100)(-80,-15)
\setlength{\unitlength}{1pt}

\put(160,0){

\put(0,0){\line(0,1){40}} \put(-2,-5){\makebox(0,0){$\zeta_0$}}
\put(-60,0){\line(0,1){40}} \put(-62,-5){\makebox(0,0){$\eta_0$}}
\put(20,0){\line(-1,2){20}} \put(30,-5){\makebox(0,0){$\delta_{1,h}$}}
\put(0,40){\line(0,1){30}} 
\put(-60,40){\line(0,1){30}} 
\put(0,40){\line(1,2){15}} \put(35,75){\makebox(0,0){$\eps_{3,h+1}$}}
\put(0,40){\line(-1,1){30}} \put(-30,75){\makebox(0,0){$\eps_{1,h}$}} 
\put(-37,40){\makebox(0,0){$\mathcal{M}$}}
\put(40,40){\makebox(0,0){$\mathcal{R}$}}
\put(-90,40){\makebox(0,0){$\mathcal{L}$}}
}

\end{picture}

\caption{\label{fig:inter6}{Interactions of a $1$-wave $\delta_1$ with $\zeta_0$ solved by the Accurate solver.}}
\end{figure}


To prove \eqref{h=k-1b-XXX-caso-J}, assume that a $1$-wave $\delta_1$ of order $h$ interacts with $\zeta_0$;
see Figure~\ref{fig:inter6} for the Accurate case. If $\delta_1$ is a rarefaction, then by Lemma~\ref{lem:intest} we have
\begin{align*}
\Delta F_h&= \Delta L_h+\Delta Q_h+\Delta L_h^0\le \frac{|\delta_1\zeta|}{2}+K_\eta^m|\eps_1\eta|-K_\eta^r|\delta_1\eta|-K_\zeta^r|\delta_1\zeta|\le\\
&\le |\delta_1|\left[\left(K_\eta^m\bigl(1+\frac{|\zeta|}{2}\bigr)-K_\eta^r\right)|\eta|+\left(1-2K_\zeta^r\right)\frac{|\zeta|}{2}\right]\,,
\end{align*}
which is nonpositive by \eqref{eq:composite2}$_{1}$ and, in particular, implies $[\Delta F_h]_-\ge (2K_\zeta^r-1)|\delta_1\zeta|/2$. Notice that in the formula above the first inequality is due to the possible presence of a wave of generation order $h$ in $\Delta L_h^0$. As a consequence, we get \eqref{h=k-1b-XXX-caso-J} by \eqref{eq:mu}:
$$
[\Delta F_{h+1}]_+=
L_{h+1}=|\eps_3|\le \frac{|\delta_1\zeta|}{2}\le \frac{1}{2K_\zeta^r-1}[\Delta F_h]_- \,.
$$
If, instead, $\delta_1$ is a shock, then
$$
\Delta F_h =\xi|\eps_3|-K_\zeta^r\xi|\delta_1\zeta|+\Delta L_h^0\le \begin{cases}
\ds\xi\left(1-2K_\zeta^r\right)\frac{|\delta_1\zeta|}{2} &\qquad \text{if $|\delta_1|\ge \rho$}\,,\\[10pt]
\ds \xi\left(C_o-2 K_\zeta^r\right)\frac{|\delta_1\zeta|}{2} &\qquad \text{if $|\delta_1|<\rho$}\,.
\end{cases}
$$
is nonpositive by \eqref{eq:composite}$_{2}$ and $[\Delta F_h]_-\ge \xi(2K_\zeta^r-1)|\delta_1\zeta|/2$. Hence, by Lemma~\ref{lem:intest}, we get
$$
[\Delta F_{h+1}]_+=\begin{cases}
\ds L_{h+1}=\xi|\eps_3|\le \xi\frac{|\delta_1\zeta|}{2}\le \frac{1}{2K_\zeta^r-1}[\Delta F_h]_- &\quad \text{if $|\delta_1|\ge \rho$}\,,\\[10pt]
\ds \Delta L_{h+1}^0= |\eps_3|\le \frac{C_o}{2}|\delta_1\zeta|\le \frac{C_o}{\xi(2K_\zeta^r-1)}[\Delta F_h]_- &\quad \text{if $|\delta_1|< \rho$}\,.
\end{cases}
$$

On the other hand, let us consider the interaction with a wave $\delta_3$ of order $h$ belonging to family $3$. We first analyze the case $\delta_3>0$; by Lemma~\ref{lem:intest}
we have $\Delta F_h=\Delta L_h+\Delta Q_h+\Delta L_h^0\le -[1+2K_\zeta^m]|\eps_1|\le 0$. Thus, by \eqref{eq:mu} we get \eqref{h=k-1b-XXX-caso-J}:
$$
[\Delta F_{h+1}]_+=\begin{cases}
\ds L_{h+1}=\xi|\eps_1|\le \frac{\xi}{1+2K_\zeta^m}[\Delta F_h]_- &\quad \text{if $|\delta_3|\ge \rho$}\,,\\[10pt]
\ds\Delta L_{h+1}^0= |\eps_1|\le \frac{1}{1+2K_\zeta^m}[\Delta F_h]_- &\quad \text{if $|\delta_3|< \rho$}\,.
\end{cases}
$$
In the other case, i.e.\ when $\delta_3$ is a shock, we have $\Delta F_h\le -\xi|\eps_1|\le 0$. Hence,
$$
[\Delta F_{h+1}]_+= \begin{cases}
\ds L_{h+1}+ Q_{h+1}=\left(1+K_\eta^m|\eta|\right)|\eps_1|\le \frac{1+K_\eta^m|\eta|}{\xi}[\Delta F_h]_- &\quad \text{if $|\delta_3|\ge \rho$}\,,\\[10pt]
\ds \Delta L_{h+1}^0=|\eps_1|\le \frac{1}{\xi}[\Delta F_h]_- &\quad \text{if $|\delta_3|< \rho$}\,.
\end{cases}
$$
Then, \eqref{h=k-1b-XXX-caso-J} is completely proved. Finally, we notice that in all above cases for $\tau\in J_h$ formula \eqref{h=k-1b-XXX} holds.
\end{proof}

The remaining analysis aims at proving that
\begin{equation}\label{eq:decFk}
\tilde{F}_k(t)=\sum_{j\ge k} F_j(t) \le {\mu^{k-1}} F_1(0)
\end{equation}
for any $k\ge 2$ and for any $t$; it is carried out as in \cite[Propositions 6.3 and 6.4]{ABCD} and follows from Proposition~\ref{prop:genorder}. In particular, formula \eqref{eq:decFk} is needed to prove that the total size of the composite waves vanishes as $\nu\rightarrow \infty$. We conclude the section by determining parameters $\rho$ and $\sigma$ as in \cite{amadori-corli-siam}. Fix $\sigma>0$ such that $\sigma=\sigma_\nu\rightarrow 0$ as $\nu\rightarrow \infty$ and estimate the total number of waves of order $<k$. Then, the total size of the composite waves is less (or equal) than
\begin{align*}
    \tilde{L}_k(t)&+\sum_{\genfrac{}{}{0pt}{}{h<k} {\tau_h<t}}
    \left(|\Delta \eta_0|+|\Delta \zeta_0|\right)(\tau_h) 
    \\
    &\le  \mu^{k-1}\cdot F_1(0)+C_o(\rho)\frac{\rho}{2}\,(|\eta|+|\zeta|)\,
    [\text{number of fronts of order $<k$}] 
    \\
    &\le \mu^{k-1}\cdot m_o+ C_o(\rho)\frac{\rho}{2}\,(|\eta|+|\zeta|)\,
    [\text{number of fronts of order $<k$}]\,,
\end{align*}
which is $<1/\nu$ by choosing $k$ sufficiently large to have the first term $\le 1/(2\nu)$ and, then, $\rho=\rho_{\nu}(m_o)$ small enough to have the second term also $\le 1/(2\nu)$.


\subsection{End of the Proof of Theorem~\ref{thm:main} and a comparison}
\label{subsec:proof}

In this last section we accomplish the proof of Theorem~\ref{thm:main} and compare the result we obtain with that proved in \cite{amadori-corli-siam,amadori-corli-source}.

\begin{proof}[End of the Proof of Theorem~\ref{thm:main}] It only remains to reinterpret the choice of the parameter $m_o$ in terms of the assumption~\eqref{hyp2} on the initial data. Notice that we can approximate the initial datum (already satisfying $1.$, $2.$ and $3.$ of Section~\ref{sec:app_sol}) in such a way that the jump $\left((p_\ell,u_\ell,\lambda_\ell),(p_m,u_m,\lambda_m)\right)$ at the interface $x=\xa$ is substituted by a jump consisting of the $2$-wave separating $(p_\ell,u_\ell,\lambda_\ell)$ and $(p_\ell,u_\ell,\lambda_m)$ at $x=\xa$ and by the solution to the newly appeared Riemann problem at $x=\xa^+$ with states $(p_\ell,u_\ell,\lambda_m)$ and $(p_m,u_m,\lambda_m)$. Analogously, we can proceed for a jump $\left((p_m,u_m,\lambda_m),(p_r,u_r,\lambda_r)\right)$ at $x=\xb$. This is possible because $p,u$ remain constant across a phase wave. Thus, we can relate hypothesis \eqref{hyp2} to \eqref{eq:boundL-bis} by including in $\bar{L}^m$ ($\bar{L}^r$, respectively) the total variation of $p_o$ and $u_o$ at the interface and as in \cite[(3.12)]{amadori-corli-siam} we can prove that
\begin{equation}\label{eq:tvpm}
\begin{aligned}
\bar{L}^\ell(0)&\le\frac{1}{2}\tv_{x<\xa}\left(\log(p_o),\frac{u_o}{a_\ell}\right)\,, \\
\bar{L}^m(0)&\le \frac{1}{2}\tv_{\xa<x<\xb}\left(\log(p_o),\frac{u_o}{a_m}\right)\,,\\
\bar{L}^r(0)&\le \frac{1}{2}\tv_{x>\xb}\left(\log(p_o),\frac{u_o}{a_r}\right) \,.
\end{aligned}
\end{equation}
Now, by \eqref{eq:cmxi}, \eqref{eq:boundL-bis} and \eqref{eq:tvpm} we have to look for an $m_o$ satisfying
\begin{equation}
    \mathcal{H}(|\eta|,|\zeta|)<\frac{1}{{c(m_o)}}-1 \,=\, \frac{2}{\cosh m_o -1} \, =: \, w(m_o) \label{hyp1}
\end{equation}
and
\begin{equation}\label{hyp2-1}
\tv_{x<\xa}\left(\log(p_o),\frac{u_o}{a_\ell}\right) + c(m_o)\tv_{\xa<x<\xb}\left(\log(p_o),\frac{u_o}{a_m}\right) + \tv_{x>\xb}\left(\log(p_o),\frac{u_o}{a_r}\right) < 2m_o\hspace{.8pt} c(m_o) \, =:\, z(m_o)\,.
\end{equation}
Notice that $w(m_o)$ is strictly decreasing from $\reali_+$ to $\reali_+$, while $z(m_o)$ is strictly increasing on the same sets.
We can now define
\begin{equation}\label{K}
\mathcal{K}(r)\, =\, z\left(w^{-1}(r) \right)\,,\qquad r\in (0,+\infty)\,,
\end{equation}
which can be written explicitly as
\begin{equation}\label{eq:K-explicit}
\mathcal{K}(r)=\frac{2}{1+r}\,c^{-1}\left(\frac{1}{1+r}\right)
=\frac{2}{1+r} \log\left(1+\frac{2}{r}\bigl(1+\sqrt{1+r}\bigr)\right)\,.
\end{equation}
Hence, if the assumption \eqref{hyp2} holds, 
it is easy to prove that one can choose $m_o$ such that \eqref{hyp1}, \eqref{hyp2-1} hold. Finally, in order to pass to the limit and prove the convergence to a weak solution, one can proceed as in \cite{Bressanbook}. Theorem~\ref{thm:main} is, therefore, completely proved.
\end{proof}

\begin{remark}\label{rem:RI} Notice that a slight improvement of Theorem~\ref{thm:main} follows from the use of the Riemann invariants
\begin{equation*}
z=u - a\log v\,,\qquad w=u+a\log v\,,
\end{equation*}
where $a=a_\ell$, $a=a_m$ and $a=a_r$ in the regions $\mathcal{L}$, $\mathcal{M}$ and $\mathcal{R}$, respectively.
Indeed, recalling Definition~\ref{eq:strengths}, one easily finds that the solution to the Riemann problem with
$U_-=(v_-,u_-,\lambda)$ and $U_+=(v_+,u_+,\lambda)$ satisfies
\begin{equation*}
|\eps_1| + |\eps_3|\le \frac 1 {4a}  \left( |w_+-w_-| +  |z_+-z_-| \right) \le   \frac 1 {2}|\log(p_+)-\log(p_-)| + \frac 1 {2a} |u_+-u_-| \,,
\end{equation*}
(with obvious notation) and the second inequality is possibly strict (for instance if the solution to the Riemann problem is a single rarefaction).

Hence, the right sides of \eqref{eq:tvpm} could be replaced by $\frac{1}{4a_\ell}\tv_{x<\xa}\left(w_o,z_o\right)$,
$\frac{1}{4a_m}\tv_{\xa<x<\xb}\left(w_o,z_o\right)$ and
$\frac{1}{4a_r}\tv_{x>\xb}\left(w_o,z_o\right)$, respectively, and \eqref{hyp2-1} could be given in terms of these quantities,
leading to a weaker assumption on the initial data.
\end{remark}

Now, we make a comparison between Theorem~\ref{thm:main} and the main result in \cite{amadori-corli-siam}, which was proved to be equivalent to Theorem $3.1$ of \cite{amadori-corli-source}. First, notice that condition $(3.6)$ of \cite[Theorem 3.1]{amadori-corli-source} can be written as $|\eta|+|\zeta|<1/2$, when applied to the current problem. Then, it implies \eqref{stab}, since
$$
\max\left\{\bigl(1+\frac{|\zeta|}{2}\bigr)\frac{|\eta|}{2},
    \bigl(1+\frac{|\eta|}{2}\bigr)\frac{|\zeta|}{2}\right\}<\frac{|\eta|+|\zeta|}{2}<\frac{1}{4}\,,
$$
i.e.\ the domain $\mathcal{D}$ contains entirely that of \cite{amadori-corli-source}.

Next, we claim that $\mathcal{H}(|\eta|,|\zeta|)<|\eta|+|\zeta|$ when $|\eta|+|\zeta|<1/2$. Indeed, by \eqref{stab} we have that
$$
\frac{|\zeta|}{1-(1+{|\zeta|}/{2}){|\eta|}/{2}}\le|\eta|+|\zeta|
$$
is equivalent to
$$
(|\eta|+|\zeta|)\bigl(1+\frac{|\zeta|}{2}\bigr)\le2\,,
$$
which holds true by \eqref{stab}. Similarly, the inequality
$$
\frac{|\eta|}{1-(1+{|\eta|}/{2}){|\zeta|}/{2}}\le|\eta|+|\zeta|\,,
$$
is equivalent to $(|\eta|+|\zeta|)(1+|\eta|/2)\le2$,
which holds true for the same reason. This proves the claim. 

On the other hand, condition $(3.7)$ of \cite[Theorem 3.1]{amadori-corli-source} here becomes by \eqref{eq:case}
\begin{equation}\label{acs37}
\tv\left(\log(p_o),\frac{1}{a_m}
u_o\right) < H(|\eta|+|\zeta|)\,,
\end{equation}
where the function $H(r)$
is only defined for $r<1/2$ by
\begin{equation}\label{H}
H(r)= 2(1-2r)k^{-1}(r)\,, \qquad k(m_o) = \frac{1-\sqrt{d(m_o)}}{2-\sqrt{d(m_o)}}\,.
\end{equation}
Here above, $d(m_o)$ is the damping coefficient introduced in \cite[Lemma 5.6]{amadori-corli-siam}. From \cite{ABCD} we already know that $\mathcal{K}>H$ in the common range $|\eta|+|\zeta|<1/2$. Thus, \eqref{hyp2} improves \eqref{acs37}, since the left-hand side of \eqref{acs37} is bigger than the left-hand side of \eqref{hyp2} and
$$
\mathcal{K}\left(\mathcal{H}(|\eta|,|\zeta|)\right)>\mathcal{K}(|\eta|+|\zeta|)>H(|\eta|+|\zeta|)\,.
$$
Consequently, we obtain enhanced conditions on the initial data in comparison with  \cite{amadori-corli-siam,amadori-corli-source},
even though the latter results apply to a wider class of $\lambda_o$.

\begin{remark}\label{rem:hyp2}
As already mentioned, in absence of one of two phase-waves, i.e.\ for example $\eta=0$, hypothesis \cite[(2.3)]{ABCD}, as well as
\eqref{eq:hyp2improved}, can be improved by \eqref{eq:hyp2improved2}. Indeed, consider a modified functional $\hat{F}=F^m+F^r+L^0$,
where $F^{m,r},L^0$ are the same of \eqref{F} with $\eta=0$ and $\mathcal{M}=\{(x,t):\, x<b\}$.
Arguing as in the proof of Proposition~\ref{global}, if we assume that
$$
c(m_o)<\frac{1}{1+|\zeta|} \qquad \text{and} \qquad c(m_o)\hspace{0.8pt}\bar{L}^m(0) +\bar{L}^r(0) \le m_o\hspace{0.8pt} c(m_o)\,,
$$
then $\hat{F}$ decreases. This last inequality holds true if we take
$$
c(m_o)\tv_{x<\xb}\left(\log(p_o),\frac{u_o}{a_m}\right) + \tv_{x>\xb}\left(\log(p_o),\frac{u_o}{a_r}\right) < z(m_o)\,,
$$
which is implied by \eqref{eq:hyp2improved2}.
\end{remark}

\appendix
\section{On the estimate \eqref{eq:intest2}}\label{appB}
\setcounter{equation}{0}

In this short appendix we consider estimate \eqref{eq:intest2}, in the case of interaction of a $1$-shock with a composite wave $\delta_0$. The question is whether we can substitute $C_o$ with $1$. The answer is only partially positive as we show in the following lemma, which improves \eqref{eq:intest2}. For simplicity, we focus on the case $\delta>0$.

\begin{lemma}\label{lem:bestC0}
    Assume that a $1$-shock $\delta_1$ interacts with $\delta_0$, when $\delta>0$
    and $|\delta_1|<\rho$. Then,
    \begin{equation}
	\label{eq:iesimpl2bis}
	|\eps_0-\delta_0|=|\eps_3|\le
	\ds\frac{1}{2}\delta|\delta_1| \quad \text{if }0<\delta\le\sqrt{5}-1\,,
    \end{equation}
    where the bound on $\delta$ in \eqref{eq:iesimpl2bis} is sharp.
    We also have
    \begin{equation}
	\label{eq:iesimpl2bis3}
	|\eps_0-\delta_0|=|\eps_3|\le
	\ds\frac{\Theta(\delta,\delta_{1})}{2}\,\delta|\delta_{1}|\,,
    \end{equation}
    for a suitable function $\Theta$ such that $\Theta(\delta,z)\le 1$ for $\delta\le2/3$, $\Theta(\delta,z)> 1$ on a right neighborhood of $z=0$ for $\delta>2/3$
    and $\lim_{z\to+\infty}\Theta(\delta,z)=0$.
\end{lemma}

Remark that conditions \eqref{eq:iesimpl2bis} and \eqref{eq:iesimpl2bis3} overlap when $\delta\le2/3$.

\begin{proof}[Proof of Lemma~\ref{lem:bestC0}]
First, we claim that for every fixed $\delta>\sqrt{5}-1$ there exist
    $r=r(\delta)$ and $R=R(\delta)$ such that
    \begin{equation}
	\label{eq:iesimpl2bis2}
	 \sup_{|\delta_{1}|<r}\left\{\frac{|\eps_0-\delta_0|}{\delta|\delta_1|/2}\right\}>1\,,
	\qquad
	 \sup_{|\delta_{1}|>R}\left\{\frac{|\eps_0-\delta_0|}{\delta|\delta_1|/2}\right\}<1\,.
    \end{equation}
Recalling that in this case $\eps_{1}=\eps_{3}+\delta_{1}$ and
$\eps_{3},\delta_{1}<0$, equation \eqref{eq:delta1eq}
can be written as
$$
k|\eps_{3}|+\sinh(|\eps_{3}|+|\delta_{1}|)-k\sinh |\delta_{1}|=0\,,
$$
i.e.\ $\widetilde{\Gamma}(k,x,y)=0$ where $x=|\delta_{1}|$, $y=|\eps_{3}|$, $k=(2+\delta)/(2-\delta)$,
$\widetilde{\Gamma}(k,x,y)=ky+\sinh(x+y)-k\sinh x$; this equation implicitly defines $y=y(x;k)$.
Then, \eqref{eq:iesimpl2bis} for a fixed $\delta$ is
equivalent to $y\leq\frac{k-1}{k+1}x$, and, since $\widetilde{\Gamma}_{y}=k+\cosh(x+y)>0$, this is
true iff $\widetilde{\Gamma}(k,x,\frac{k-1}{k+1}x)\ge0$ for all $x\ge0$. It holds
\begin{align*}
    \widetilde{\Gamma}\left(k,x,\frac{k-1}{k+1}x\right)
    &=\left(k\frac{k-1}{k+1}+\frac{2k}{k+1}-k\right)x+
    \sum_{h=1}^{\infty}\left[\left(\frac{2k}{k+1}\right)^{2h+1}-k\right]
    \frac{x^{2h+1}}{(2h+1)!}\\
    &=:\sum_{h=1}^{\infty}a_{k,h}\frac{x^{2h+1}}{(2h+1)!}\,,
\end{align*}
as well as
$$
\widetilde{\Gamma}\left(k,x,\frac{k-1}{k+1}x\right)
=a_{k,1}x^{3}+o(x^{3})=-\frac{k(k-1)(k^{2}-4k-1)}{6(k+1)^{3}}x^{3}+o(x^{3})\,.
$$
We have $a_{k,1}<0$ iff $k>2+\sqrt{5}$; hence, $\widetilde{\Gamma}(k,x,\frac{k-1}{k+1}x)<0$ in a
right neighborhood of $x=0$. Since $k>2+\sqrt{5}$ iff $\delta>\sqrt{5}-1$,
\eqref{eq:iesimpl2bis2}$_{1}$ follows.
Moreover, for $k>1$ fixed, there exists $\tilde{h}(k)$ such that
$a_{k,h}>0$ for $h\ge\tilde{h}(k)$; this implies
$\lim_{x\to+\infty}\widetilde{\Gamma}(k,x,\frac{k-1}{k+1}x)=+\infty$,
hence \eqref{eq:iesimpl2bis2}$_{2}$. This proves the claim. As a consequence, for
    $\delta>\sqrt{5}-1$ it holds
    $$
    \sup_{|\delta_{1}|\neq0}\left\{\frac{|\eps_0-\delta_0|}{\delta|\delta_1|/2}\right\}>1\,,
    $$
    and the estimate on the left in \eqref{eq:iesimpl2bis} fails.

On the contrary, when $k\le2+\sqrt{5}$ (i.e.\ $\delta\le\sqrt{5}-1$),
there holds $0\le a_{k,1}<a_{k,h}$ for $h>1$, hence
$\widetilde{\Gamma}(k,x,\frac{k-1}{k+1}x)>0$ for all $x\ge0$ and
\eqref{eq:iesimpl2bis} follows.

Now, we prove \eqref{eq:iesimpl2bis3}; by \eqref{eq:delta1eq}
and the Mean Value Theorem, there exists $s\in\,]\eps_{3}+\delta_{1},\delta_{1}[$ such that $\Gamma(\eps_3+\delta_1)-\Gamma(\delta_1)=\Gamma'(s)\eps_{3}$.
Since $s<\delta_{1}<0$ we get $\Gamma'(s)>\Gamma'(\delta_{1})=k+\cosh\delta_{1}$ so that $(k+\cosh|\delta_{1}|)\cdot|\eps_{3}|\le(k-1)\sinh|\delta_1|$.
Then, \eqref{eq:iesimpl2bis3} holds for all $\delta>0$, since
\begin{align*}
    |\eps_{3}|\le\frac{k-1}{k+\cosh|\delta_{1}|}\sinh|\delta_1|
    =\left(\frac{k+1}{k+\cosh|\delta_{1}|}\cdot\frac{\sinh|\delta_1|}{
    |\delta_{1}|}\right)\cdot\frac{1}{2}|\delta\delta_{1}|
    =:\Theta(\delta,|\delta_{1}|)\cdot\frac{1}{2}|\delta\delta_{1}|\,.
\end{align*}

Finally, let us prove the properties of $\Theta$.
For simplicity, call $z=|\delta_{1}|$; then $\Theta(\delta,z)\leq 1$ iff $0\le \left(k+\cosh z\right) -(k+1)\sinh z/z$, which is equivalent to
$$
\sum_{h=1}^{\infty}\left[(2h+1)-(k+1)\right]\frac{z^{2h}}{(2h+1)!}\ge 0\,.
$$
This last inequality is verified for every $z\geq0$ iff $2h+1\geq k+1$
for every $h\geq1$, i.e.\ iff $k\le 2$. It is easy to check that $k=2$
is equivalent to $\delta=2/3$.
This implies that $\Theta(\delta,z)\le 1$ for every $z\ge0$ if
$\delta\le2/3$, while $\Theta(\delta,z)> 1$ on a right neighborhood of $0$
if $\delta>2/3$.
By construction $\Theta(\delta,\delta_{1})< C_o(\delta_{1})\le C_o(\rho)$, with
$\lim_{\delta_{1}\to\infty}\Theta(\delta,\delta_{1})=0$, while
$\lim_{\delta_{1}\to\infty}C_o(\delta_{1})=+\infty$. Hence,
\eqref{eq:iesimpl2bis3} is far better than \eqref{eq:intest2} especially
when $\delta>2/3$.
\end{proof}


{\small
\bibliographystyle{abbrv}

}


\end{document}